\documentclass[10pt, a4paper]{amsart}

\usepackage[T1]{fontenc}
\usepackage[utf8]{inputenc}
\usepackage{amsmath, amssymb, amsfonts, amsthm}
\usepackage[margin=2.0cm]{geometry}
\usepackage[]{setspace}
\usepackage[usenames,dvipsnames]{color}
\numberwithin{equation}{section}
\usepackage{dirtytalk}
\usepackage{enumitem}
\usepackage{tikz-cd}
\usepackage{mathrsfs}
\usepackage{xcolor}
\usepackage[toc,page]{appendix}
\usepackage[, unicode]{hyperref}
\hypersetup{
   colorlinks,
   citecolor=olive,
   filecolor=black,
   linkcolor=blue,
   urlcolor=blue
 }%
\usepackage{mathtools,halloweenmath}
\usepackage{mathbbol}
\usepackage{todonotes}
\theoremstyle{plain}
\newtheorem{propo}{Proposition}[section]
\newtheorem{lema}[propo]{Lemma}
\theoremstyle{plain}
\newtheorem{teorema}[propo]{Theorem}
\newtheorem{cor}[propo]{Corollary}
\newtheorem{fact}[propo]{Fact}

\theoremstyle{definition}

\newtheorem{defin}[propo]{Definition}

\newtheorem{remark}[propo]{Remark}

\newcommand{\ol}[1]{\overline{#1}}

\renewcommand{\.}{\ldots}

\newcommand{\dl}{ \Big( \kern-0.6em{ \Big( }}
\newcommand{\dr}{ \Big) \kern-0.6em{ \Big) }}

\newcommand{\sqdl}{ \Big[ \kern-0.55em{ \Big[ }}
\newcommand{\sqdr}{ \Big] \kern-0.55em{ \Big] }}

\DeclareMathOperator{\img}{im} 
\DeclareMathOperator{\Hom}{Hom} 
\DeclareMathOperator{\End}{End} 
\DeclareMathOperator{\Frac}{Frac} 
\renewcommand{\char}{char} 
\DeclareMathOperator{\dprk}{dp-rk} 
\DeclareMathOperator{\Sub}{Sub} 
\DeclareMathOperator{\Dir}{Dir} 
\DeclareMathOperator{\Gl}{GL} 
\DeclareMathOperator{\res}{res} 
\DeclareMathOperator{\val}{val} 
\DeclareMathOperator{\dd}{\partial} 
\DeclareMathOperator{\Mod}{Mod} 
\DeclareMathOperator{\Vect}{Vect} 
\DeclareMathOperator{\gres}{\widehat{res}} 
\DeclareMathOperator{\tr}{tr} 
\DeclareMathOperator{\wt}{wt} 
\DeclareMathOperator{\Stab}{Stab} 

\def\a{\alpha}
\def\b{\beta}
\def\g{\gamma}
\def\G{\Gamma}
\def\d{\delta}

\newcommand{\e}{\varepsilon}
\newcommand{\f}{\varphi}

\renewcommand{\t}{\tau}
\def\T{\Theta} 
\let\lw\l
\renewcommand\l{{\lambda}}
\def\L{\Lambda}
\def\k{\kappa}
\def\s{\varsigma}

\def\E{\exists} 
\def\A{\forall} 

\newcommand{\Sii}{\Longleftrightarrow}

\def\N{\mathbb{N}}

\def\F{\mathbb{F}}

\def\KK{\mathbb{K}}

\def\p{\mathfrak{p}} 
\def\O{\mathcal{O}} 
\def\M{\mathscr{M}} 
\def\Ca{\mathscr{C}}
\def\Da{\mathscr{D}}

\def\Ind#1#2{#1\setbox0=\hbox{$#1x$}\kern\wd0\hbox to 0pt{\hss$#1\mid$\hss}
\lower.9\ht0\hbox to 0pt{\hss$#1\smile$\hss}\kern\wd0}

\def\Notind#1#2{#1\setbox0=\hbox{$#1x$}\kern\wd0\hbox to 0pt{\mathchardef
\nn=12854\hss$#1\nn$\kern1.4\wd0\hss}\hbox to
0pt{\hss$#1\mid$\hss}\lower.9\ht0 \hbox to
0pt{\hss$#1\smile$\hss}\kern\wd0}

\usepackage{fancyhdr}
\usepackage{csquotes}
\usepackage{faktor}
\usepackage[english]{babel}
\newtheoremstyle{named}{}{}{}{}{}{}{.5em}{\normalfont{\thmnumber{#2}} \itshape\thmnote{#3.}}
\theoremstyle{named}
\newtheoremstyle{asd}{}{}{}{}{\scshape}{}{.5em}{\scshape{\thmname{#1}} \normalfont{\thmnumber{#2}}}
\theoremstyle{asd}
\fancypagestyle{mypagestyle}{%

\lhead{$\,$}
\rhead{\Cross}
}
\fancypagestyle{plain}{
\fancyfoot{}
\fancyfoot[LE,RO]{\thepage}
}
\usepackage{titlesec}
\titleformat{\section}
  {\centering\Large\scshape}{\scshape\thesection}{1em}{}
\titleformat{\subsection}{\large\bfseries}{\scshape\thesubsection}{0.5em}{}
\allowdisplaybreaks

\title[RUNNING TITLE]{\Large Fields of dp-Rank 2 and their \texorpdfstring{$W_2$}{}-Topologies: The Characteristic 2 Case}
\author{PAULO ANDRÉS SOTO MORENO}
\address{Universit\'{e} Paris Cit\'{e}, Sorbonne Universit\'{e}, CNRS, IMJ-PRG, F-75013 Paris, France}
\email{paulo.soto@imj-prg.fr}

\begin{document}

\begin{abstract}
   
   In this note we reproduce Johnson's analysis of $W_2$-topologies on fields of characteristic 2, which was originally stated for fields of characteristic different than 2. Following his framework, we prove that the canonical topology of an unstable field of characteristic 2 and dp-rank 2 is a $V$-topology. Additionally, we show that any $W_2$-topology on a field of characteristic 2 is either induced by the intersection of two valuation rings or it is induced by dense pre-diffeo-valuation data, completing the picture for all positive characteristic fields. 
\end{abstract}
\maketitle

\section{Introduction}

Shelah's Conjecture predicts that any NIP field is either finite, separably closed, real closed, or admits a non-trivial henselian valuation. In a recent series of outstanding papers, Will Johnson proved the conjecture in the important special case of 
fields $\KK$ of finite \emph{dp-rank}, by a comprehensive study along two axes: lattices of type-definable linear sub-spaces of powers of $\KK,$ and definable field topologies of $\KK$ and their \emph{$V$-topological} coarsenings, i.e.~a topological coarsening that is induced either by an absolute value or by a valuation on $\KK.$ The latter point of view is of special interest, because Shelah's Conjecture is a consequence of existence and uniqueness of \emph{definable $V$-topologies} in infinite non-separably closed NIP fields. More precisely, the proof of \cite[Proposition 6.4]{dp4} shows that if an infinite non-separably closed field is NIP, then it is real closed or henselian if it admits a unique definable $V$-topology ---see the proof of Proposition 2.41 of the survey document \cite{anscombe} for more details. In \cite{dp2}, Johnson identified the \emph{canonical topology} of a dp-finite field $\KK$ as the field topology in which basic neighborhoods of $0$ are of the form $X-X$ where $X\subseteq \KK$ is definable and $\dprk(X)=\dprk(\KK).$ 
Afterwards, in \cite{dp5}, he defined a notion of \emph{weight} associated to a ring topology, and defined the class of \emph{$W_n$-topologies} of a field $\KK$ as the set of ring topologies on $\KK$ that have weight at most $n,$ for $n\geq 1.$ Some of the most important features of weights and canonical topologies are the following, cf.~\cite[Theorem 4.10]{dp5} and \cite[Fact 1.7]{dp6}.
\begin{itemize}[wide]
    \item A field topology is a $W_1$-topology if and only if it is a $V$-topology,

    \item The canonical topology of an unstable dp-finite field $\KK$ is definable and its weight is bounded above by $\dprk(\KK),$

    \item The number of $V$-topological coarsenings of a $W_n$-topology is at least 1 and at most $n.$

    \item Any $V$-topology on a dp-finite field is definable if and only if it is a coarsening of the canonical topology.
\end{itemize}

If (the theory of) a field is dp-finite and stable, then it has finite Morley rank and thus it is either finite or algebraically closed, cf.~\cite[Corollary 2.5]{dp2} and \cite[Theorem 1]{macintyre}, so the focus of the classification relies on \emph{unstable} dp-finite fields. 
As a consequence of these statements, if the canonical topology of an unstable dp-finite field is a $V$-topology, then it is in fact the unique definable $V$-topology, and the classification holds. This is the content of the \emph{Valuation Conjecture}: the canonical topology of any unstable dp-finite field is a $V$-topology. It turns out that this conjecture is false, as proven in \cite[Theorem 10.1]{dp4}. Roughly speaking, Johnson constructs an algebraically closed valued field with a derivation, of residue characteristic 0, having dp-rank 2 and whose canonical topology is not a $V$-topology. However, the Valuation Conjecture for fields of dp-rank 2 and \emph{odd} characteristic does hold, as seen in \cite[Theorem 6.29]{dp4}. Moreover, if the characteristic is not 2, Johnson even classifies $W_2$-topologies as the field topologies induced by either an intersection of at most two independent valuation sub-rings, or induced by some differential structure called \emph{dense diffeo-valuation data}, cf.~\cite[Theorem 8.5]{dp5}. This classification theorem is stated for fields of characteristic 0, but its proof holds as well for fields of characteristic different than 2, as observed by Johnson in the comment before \cite[Theorem 1.7]{dp4}.

In this document we give a positive solution to the Valuation Conjecture for fields of dp-rank 2 and characteristic 2, answering a question of Johnson. In the process, we state a classification theorem for $W_2$-topologies on fields of characteristic 2, analogous to Johnson's classification for fields of characteristic not 2. Putting the odd and the even cases together, the main theorems of this document are the following. 

\begin{teorema}[Theorem {\ref{teo1}} for $p=2$ and {\cite[Theorem 6.29]{dp4}} for $p>2$]
    If $K$ is a field of characteristic $p>0$ and of dp-rank at most 2, then $K$ is either stable or its canonical topology is a $V$-topology.
\end{teorema}

\begin{teorema}[Theorem {\ref{classw2}} for $p=2$ and {\cite[Theorem 8.5]{dp5}} for $p>2$ or $p=0$]
    Let $\t$ be a field topology on a field of characteristic $p$. Then $\t$ is a $W_2$-topology if and only if exactly one of the following statements holds:
    \begin{enumerate}[wide]
        \item $\t$ is a $V$-topology.

        \item $\t$ is a strict $V^{2}$-topology.

        \item $\t$ is a $DV$-topology.
    \end{enumerate}
\end{teorema}

Thanks to a personal communication with Will Johnson, we learned that the Valuation Conjecture is in fact true for \emph{any} unstable dp-finite field of positive characteristic, following the original work of Yang Yang's master thesis \cite{yang}. This thesis, however, has not been published at the time of writing of this note.
 
This document closely follows Johnson's analysis of 2-inflators and $W_2$-topologies in \cite{dp4} for fields of odd characteristic, modified for fields of characteristic 2. Inflators were one of the main tools that Johnson used to study dp-finite fields $\KK$, as they appear naturally when studying lattices of type-definable vector subspaces of powers of $\KK.$   
For preliminaries of inflators, dense diffeo-valued fields and $W$-topologies, see \cite{dp3}, \cite{dp4} and \cite{dp5} respectively. In Section \ref{sec2} we gather some definitions and facts about inflators, and we copy verbatim a result of \cite{dp4} that holds in characteristic 2 in order to obtain a formula for a 2-inflator $\s$ that satisfies the so-called \emph{Strong Assumptions,} see Definition \ref{str}. In Section \ref{sec3}, we see how the integral closure of the fundamental ring $R$ of $\s$ is a valuation ring, and study its structure. In Section \ref{sec4}, we use the generalized residue map displayed in Section \ref{sec2} to endow $R$ with a differential structure, and we study its topological and valuative properties. In Section \ref{sec5} we apply this theory to the canonical topology of an unstable field of dp-rank at most 2, and prove our first main theorem. 
In Section \ref{sec6}, we introduce \emph{dense pre-diffeo-valued} fields and study their algebra and topology. In particular, we show that dense pre-diffeo-valuation data can be recovered by any 2-inflator satisfying the Strong Assumptions, which allows us to classify $W_2$-topologies for fields of even characteristic.

\section{Preliminaries}
\label{sec2}

\subsection{Directories and Inflators}

In this section we will summarize the definitions of the objects we are going to use throughout the document, as presented in \cite{dp3}. The reader already familiar with directories, inflators, mutation and malleability can skip ahead to the next section. In what follows, let $k_0$ be an auxiliary field. The reader is welcome to assume that $k_0$ is infinite. 
We call an abelian category $\Ca$ \emph{$k_0$-linear} if for any pair of objects $A,B$ of $\Ca,$ the set of morphisms $\Hom_\Ca(A,B)$ is a vector space over $k_0.$ If $\Ca$ is a $k_0$-linear abelian category and if $A$ is an object thereof, two monomorphisms $X\hookrightarrow A,Y\hookrightarrow A$ are called \emph{equivalent} if there is an isomorphism $X\to Y$ making the diagram 
\begin{center}
    \begin{tikzcd}
X \arrow[rd] \arrow[r, hook] & A                 \\
                             & Y \arrow[u, hook]
\end{tikzcd}
\end{center}
commute. A \emph{subobject} of $A$ is the equivalence class of a monomorphism $X\hookrightarrow A.$ There is a partial order on subobjects of $A$ given by $[X\hookrightarrow A]\leq[Y\hookrightarrow A]$ if and only if there is a monomorphism $X\hookrightarrow Y$ such that the diagram 
\begin{center}
    \begin{tikzcd}
X \arrow[rd, hook] \arrow[r, hook] & A                 \\
                                   & Y \arrow[u, hook]
\end{tikzcd}
\end{center}
commutes.
The poset $\Sub_\Ca(A)$ of subobjects of $A$ admits a structure of bounded lattice, with bottom element $[0\hookrightarrow A],$ top element $[\operatorname{Id}:A\to A],$ meet defined by $[X\hookrightarrow A]\wedge [Y\hookrightarrow A]=[X\times_AY\hookrightarrow A]$ and join defined by $[X\hookrightarrow A]\vee [Y\hookrightarrow A]=[X\oplus Y\hookrightarrow A].$ Note that if $\Ca$ is the category $S$Mod of modules over a ring $S,$ then subobjects of $A$ correspond exactly to $S$-submodules of $A$, and the lattice operations $(\bot,\top,\vee,\wedge)$ correspond exactly to $(0,A,+,\cap).$  
For a given $n>0,$ the poset $\Sub_\Ca(A^{n})$ also admits an action of $\Gl_n(k_0)$ defined as follows. Let $\phi:X\to A^{n}$ be a monomorphism, and let $\pi_i:A^{n}\to A$ be the projection to the $i^{th}$ coordinate. Then there are morphisms $\ol{\pi}=(\pi_1\circ \phi,\.,\pi_n\circ \phi)\in\Hom_\Ca(X,A)^{n}=\Hom_\Ca(X,A^{n}).$ By $k_0$-linearity, $\mu\cdot\ol{\pi}\in\Hom_\Ca(X,A^{n}),$ so we may define $\mu\cdot [X\hookrightarrow A^{n}]$ as the equivalence class of the image of $\mu\cdot\ol{\pi}.$         
The following fact implies that the lattice structure on the subobjects of $A$ is in fact \emph{modular,} meaning that for any pair of elements $x,y$ of the lattice, the maps $-\vee y:[x\wedge y,x]\to[y,x\vee y]$ and $x\wedge-:[y,x\vee y]\to [x\wedge y,x]$ are isomorphisms of bounded lattices.  

\begin{fact}[Mitchel's Embedding Theorem]
\label{mitchel}
    Let $\mathscr{C}$ be a $k_0$-linear abelian category. Then there is a $k_0$-algebra $S$ and a fully faithful exact functor $F:\mathscr{C}\to S\Mod.$ Equivalently, $\mathscr{C}$ is a full subcategory of $S\Mod$ for some $k_0$-algebra $S.$  
\end{fact}

Modularity is now seen to be a consequence of the so called \emph{diamond} isomorphism theorem: if $N_1,N_2$ are two sub-modules of a module $M,$ then $(N_1+N_2)/N_1\cong N_2/(N_1\cap N_2).$ 
Therefore it makes sense to refer to the \emph{length} $l(A)$ of an object $A$ of a $k_0$-linear abelian category $\Ca,$ defined as the length of the sub-object lattice $\Sub_\Ca(A).$\footnote{In general, the length of a lower bounded modular lattice $M$, when seen as a poset, is the maximal $n\in\N$ such that there is a chain of elements $x_0=\bot<x_1<\.<x_{n}$ in $M,$ or $\infty$ if there is no such maximum. An element $a$ of $M$ has \emph{finite length} if the lattice $[\bot,a]=\{b\in M:\bot\leq b\leq a\}$ has finite length.}
We say that $A$ has \emph{finite length} if $\Sub_\Ca(A)$ has finite length, and call $A$ \emph{semi-simple} if $A$ is a \emph{finite} direct sum of \emph{simple} objects, i.e.~of objects of length $1.$ 

\begin{defin}
     Let $\Ca$ be a $k_0$-linear abelian category and let $A$ be an object thereof. We define the \emph{directory of $A$} as the collection of modular lattices $$\Dir_{\Ca}(A)=(\Sub_\Ca(A),\Sub_\Ca(A^{2}),\Sub_\Ca(A^{3}),\.)$$ together with the lattice structure on each $\Sub_\Ca(A^{n}),$ the operation $$\oplus:\Sub_\Ca(A^{n})\times \Sub_\Ca(A^{m})\to\Sub_\Ca(A^{n+m})$$ given by $[X\hookrightarrow A^{n}]\oplus[Y\hookrightarrow A^{m}]:=[X\oplus Y\hookrightarrow A^{n+m}],$ and the action of $\Gl_n(k_0)$ on each $\Sub_{\Ca}(A^{n}).$
\end{defin}

If $\mathscr{\Ca}$ is the category $K$Vect of vector spaces over $K,$ or $S$Mod of modules over a ring $S,$ we write $\Dir_K(A)$ and $\Dir_S(A)$ respectively. We say that $\Dir_\Ca(A)$ is \emph{semi-simple} if $A$ is semi-simple. 

\begin{fact}[Artin-Wedderburn, cf.~{\cite[Theorem 2.7]{dp3}}]
Let $A$ be an object of an abelian category $\Ca.$ The following statements are equivalent:
\begin{enumerate}[wide]
    \item $\Dir_\Ca(A)$ is semi-simple.
    \item $\Dir_\Ca(A)\cong\Dir_S(M)$ for some semi-simple ring $S$ and some finitely generated $S$-module $M.$
    \item $\Dir_\Ca(A)\cong\Dir_S(S)$ for some semi-simple ring $S.$
\end{enumerate}
\end{fact}

\begin{defin}
    Let $\Ca,\Da$ be two $k_0$-linear abelian categories and let $A,B$ be two objects of $\Ca$ and $\Da$ respectively. A system of functions $\s:\Dir_\Ca(A)\to \Dir_\Da(B)$ given by $\s_n:\Sub_\Ca(A^{n})\to\Sub_\Da(B^{n})$ for each $n>0,$ is called a \emph{directory morphism} if:
    \begin{enumerate}[wide]
        \item For any $n>0,$ the map $\s_n$ is monotone: $X\leq Y$ implies that $\s_n(X)\leq\s_n(Y)$ for any pair of subobjects $X,Y\in\Sub_\Ca(A^{n}).$
        \item For any pair $n,m>0$ and any pair of subobjects $X\in\Sub_\Ca(A^{n})$ and $Y\in\Sub_\Ca(A^{m}),$ $$\s_{n+m}(X\oplus Y)=\s_n(X)\oplus\s_m(Y).$$
        \item For any $n>0,$ any $\mu\in\Gl_n(k_0)$ and any $X\in\Sub_\Ca(A^{n}),$ $$\s_n(\mu\cdot X)=\mu\cdot\s_n(X).$$
    \end{enumerate}
    We say that two directory morphisms $\s:\Dir_\Ca(A)\to\Dir_\Da(B),$ $\s':\Dir_\Ca(A)\to\Dir_{\Da'}(B')$ are \emph{equivalent} if there is a directory isomorphism $\Dir_\Da(B)\to\Dir_{\Da'}(B')$ making the diagram
    \begin{center}
        \begin{tikzcd}
\Dir_\Ca(A) \arrow[rrd, "\s'"'] \arrow[rr, "\s"] &  & \Dir_\Da(B) \arrow[d] \\
                                               &  & \Dir_{\Da'}(B')      
\end{tikzcd}
    \end{center}
    commute.
\end{defin}

\begin{defin}
    Let $A$ be a semi-simple object of an abelian category $\Ca$ and let $d>0$. A directory morphism $\s:\Dir_K(K)\to\Dir_\Ca(A)$ is called a \emph{$d$-inflator} if $l(\s_n(V))=d\cdot\dim_K(V)$ for any $V\leq K^{n},$ $n>0.$ We say that $\s$ is an \emph{inflator} if it is a $d$-inflator for some $d>0.$
\end{defin}

If $\a\in K$, then $\T_\a:=K\cdot(1,\a)$ denotes the line of slope $\a$ in $K^{2}.$ Analogously, if $\f\in\End_S(M)$ for some $S$-module $M,$ then $\T_\f$ will denote the graph of the endomorphism $\f,$ i.e.~$\T_\f:=\{(x,\f(x)):x\in M\}.$ If $\s:\Dir_K(K)\to\Dir_S(M)$ is an inflator and $\f\in\End_S(M),$ we say that $a\in K$ \emph{specializes to $\f$ with respect to $\s$} if $\s_2(\T_a)=\T_\f.$ 

\begin{fact}[cf.~{\cite[Proposition 5.7]{dp3}}]
\label{fundring}
    Let $\s:\Dir_K(K)\to\Dir_S(M)$ be a $d$-inflator. Then the set $R_\s=\{a\in K:\s_2(\T_a)=\T_\f\text{ for some }\f\in\End_S(M)\}$ is a $k_0$-subalgebra of $K.$ Define $\widehat{\res}:R_\s\to\End_S(M)$ by $\widehat{\res}(a)=\f$ if and only if $\s_2(\T_a)=\T_\f.$ Then $\widehat{\res}$ is a $k_0$-algebra morphism whose kernel $I_\s=\{a\in K:\s_2(\T_a)=M\oplus 0\}$ satisfies $1+I\subseteq R_\s^{\times}$ and is contained in the Jacobson radical of $R_\s.$
\end{fact}

\begin{defin}
    We call $R_\s$ the \emph{fundamental ring} and $I_\s$ the \emph{fundamental ideal} of $\s.$ The map $\widehat{\res}:R\to\End_S(M)$ is called the \emph{generalized residue map} associated to $\s.$
\end{defin}

Let $\s:\Dir_K(K)\to\Dir_S(M)$ be a $d$-inflator, where $S$ is a $k_0$-algebra and $M$ is a semi-simple $S$-module of length $d.$ Denote by $R$ its fundamental ring and by $I$ its fundamental ideal, and let $L=K\cdot(a_1,\.,a_n)$ be a line in $K^{n}.$ Then the Kronecker product $K^{m}\otimes L$ is a subobject of $K^{nm}.$ If $V\leq K^{m},$ then $V\otimes L\leq K^{m}\otimes L\leq K^{nm}.$ We may define a new map $\s':\Dir_K(K)\to \Dir_S(\s_n(L))$ by putting $$\s_m'(V)=\s_{nm}(V\otimes L).$$ 

Note that for any $m>0,$ $K^{m}\otimes L=L^{m}$ as subspaces of $K^{nm},$ so that $\s_m'(V)=\s_{nm}(V\otimes L)\leq\s_{nm}(K^{m}\otimes L)=\s_{nm}(L^{m})=(\s_{n}(L))^{m},$ meaning that the target directory is the adequate one. Since both $-\otimes L$ and $\s$ preserve $\leq,$ then $\s'$ preserves $\leq$ as well. Also, $-\otimes L$ distributes over $\oplus,$ so as $\s$ distributes over $\oplus$ as well, then so does $\s'$. Finally, if $\xi\in\Gl_m(k_0),$ then 
$$\s_m'(\xi\cdot V)=\s_{nm}((\xi\cdot V)\otimes L)=\s_{nm}((\xi\otimes \operatorname{Id}_n)\cdot(V\otimes L))=(\xi\otimes\operatorname{Id}_n)\cdot\s_{nm}(V\otimes L)=\xi\cdot\s_m'(V).$$
The last equality follows because the action of $\xi\otimes \operatorname{Id}_n$ in $M^{nm}$ is equal to the action of $\xi$ on $(M^{n})^{m},$ which can be checked by tracking the action of both matrices in a vector of $M^{nm}=(M^{n})^{m}$ respectively. 
To sum up, $\s'$ is a well-defined directory morphism. It is in fact a $d$-inflator. Indeed, if $V\leq K^{m},$ 
$$l(\s'_m(V))=l(\s_{nm}(V\otimes L))=d\cdot\dim_K(V\otimes L)=d\cdot\dim_K(V)\cdot\dim_K(L)=d\cdot\dim_K(V).$$

\begin{defin}
    The \emph{mutation along $L$} of $\s$ is the $d$-inflator $\s':\Dir_K(K)\to\Dir_S(\s_n(L))$ given by $\s_m'(V)=\s_{nm}(V\otimes L)$ for any $V\leq K^{m}.$
\end{defin}

 Finally, we define malleability. Though we will not use this property explicitly, it will be mentioned several times. We still record the definition for the sake of completeness. 
\begin{defin}
    An inflator $\s:\Dir_K(K)\to D_{\bullet}$ is called \emph{malleable} if, for every $n>0,$ every $V\subseteq W$ in $\Sub_K(K^{n})$ and every object $Y\in D_n$ such that $\s_n(V)\leq Y\leq\s_n(W),$ if $l(Y)=1+l(\s_n(V)),$ then there is some $X\in\Sub_{K}(K^{n})$ such that $V\subseteq X\subseteq W,$ $l(X)=1+l(V)$ and $\s_n(V)\leq Y\leq \s_n(X)\leq \s_n(W).$ 
\end{defin}

\subsection{Conventions}

We will relax Johnson's hypothesis of the existence of an \emph{infinite} auxiliary subfield $k_0$ such that all rings and fields mentioned in the document are $k_0$-algebras. Namely, we allow $k_0$ to be the algebraic part of $K$, i.e.~$\F_p^{alg}\cap K$ (for $\char(K)=p>0$), even when this is finite. Johnson calls an inflator $\s:\Dir_K(K)\to D_{\bullet}$ \emph{of weak multi-valuation type} if its fundamental ring $R_\s$ contains a non-zero ideal of a multi-valuation ring of $K.$ In \cite{dp4}, he studies malleable 2-inflators with no weak multi-valuation type mutations. In order to keep track of the number of maximal ideals of such multi-valuation rings, we introduce the following notation.  

\begin{defin}
\label{v2type}
    Let $\k\geq1$ be a countable cardinal. An inflator $\s:\Dir_K(K)\to D_{\bullet}$ is \emph{weakly of $V^{\k}$-type} (or of \emph{weak $V^{\k}$-type}) if $R_\s$ contains a non-zero ideal of a multi-valuation ring with at most $\k$-many maximal ideals.
\end{defin}

Note that being weakly of $V^{\aleph_0}$-type is equivalent to being weakly of multi-valuation type, because all multi-valuation rings have finitely many maximal ideals. With this notation in mind, we will work with a fix $k_0$-linear 2-inflator $\s:\Dir_K(K)\to D_{\bullet}$ satisfying the following hypotheses:
\begin{enumerate}
    \item $\s$ is malleable, and

    \item no mutation of $\s$ is of weak $V^{2}$-type.
\end{enumerate}
If $\s$ is an inflator satisfying these conditions, we say that $\s$ satisfies the \emph{Weak Assumptions}. Note that these assumptions are invariant under mutation, because a mutation of a malleable inflator remains malleable (cf.~\cite[Proposition 10.13]{dp3}), and a mutation of a mutation of $\s$ remains a mutation of $\s$ (cf.~\cite[Proposition 10.5]{dp3}). We will only use the second of the Weak Assumptions until Lemma \ref{QDENSE}. In fact, this is all that Johnson uses in \cite[Lemma 5.20]{dp4}: no \emph{valuative} ball should be contained in the fundamental ring of $\s.$ We make this clearer in Lemma \ref{QDENSE}. 
Recall that an element $\a\in K$ is called \emph{wild} (with respect to $\s$) if its \emph{tame locus} $$S_\a=\{\a\}\cup\left\{\dfrac{1}{\a-q}:q\in k_0\right\}$$ is disjoint from $R_\s.$ If this is not the case, we call $\a$ \emph{tame} (with respect to $\s$), in which case all but at most 2 elements of $S_\a$ are in $R_\s,$ cf.~\cite[Lemma 5.22]{dp3}.
We also recall that, whenever $k_0$ is infinite, then any inflator $\xi$ on $K$ is of multi-valuation type if and only if all elements of $K$ are tame with respect to $\xi.$ If we allow $k_0$ to be finite, we still have the following result.

\begin{lema}[Cf.~{\cite[Proposition 5.25]{dp3}}]
    Let $\s:\Dir_K(K)\to D_{\bullet}$ be a 2-inflator. If all elements of $K$ are tame with respect to $\s,$ then $R_\s$ is a multi-valuation ring with at most 2 maximal ideals. 
\end{lema}
\begin{proof}
    Since $R_\s$ is a $k_0$-algebra, by \cite[Lemma 5.24]{dp3}, it is enough to show that $$\left(\{\a\}\cup\left\{\dfrac{1}{\a},\dfrac{1}{\a-1}\right\}\right)\cap R_\s\neq\emptyset$$ for any $\a\in K.$ But all but at most 2 elements of $S_\a$ are in $R_\s,$ so at least one of these three listed elements is indeed in $R_\s.$ Thus $R_\s$ is a multi-valuation ring having at most 2 maximal ideals.
\end{proof}

\begin{remark}
\label{wildex}
    Any mutation $\s'$ of $\s$ admits wild elements. Otherwise, $R_{\s'}$ would be a multi-valuation ring having at most 2 maximal ideals, i.e.~$\s'$ would be of $V^{2}$-type.   
\end{remark}

 We will also use the following fact, which is puts in evidence the independence on the cardinality of $k_0$. We include a different proof for the sake of completeness.

\begin{fact}[Cf.~{\cite[Lemma 10.11]{dp3}}]
\label{tamingwild}
    Let $\s:\Dir_K(K)\to\Dir_S(M)$ be a 2-inflator, and let $\a\in K$ be wild with respect to $\s.$ Let $\s'$ be the mutation of $\s$ along $K\cdot(1,\a).$ Then $\a$ is tame with respect to $\s'.$ 
\end{fact}
\begin{proof}
    Suppose $\a$ is wild with respect to $\s',$ so that $S_\a\cap R_{\s'}=\emptyset.$ By \cite[Lemma 10.10]{dp3}, there are non-zero elements $\e_1,\e_2$ of $M$ such that $(\e_1,0,0),(\e_2,\e_2,\e_2)\in\s_3(\{(x,\a x,\a^{2}x):x\in K\}),$ corresponding to the fact that $1/\a$ and $1/(\a-1)$ are not elements of $R_{\s'}.$ 
    In the same fashion, 
    \begin{align*}
        \a\not\in R_{\s'}&\Sii \s_2'(K\cdot(1,\a))\text{ is not the graph of any }\f\in\End_S(\s_2(K\cdot(1,\a)))\\
        &\Sii \s_2'(K\cdot(1,\a))\cap\Big(0^{2}\oplus\s_2(K\cdot(1,\a))\Big)\neq 0^{2}\oplus 0^{2},
    \end{align*} so there is a nontrivial vector $(x_1,x_2)\in\s_2(K\cdot(1,\a))\subseteq M^{2}$ satisfying that $(0,0,x_1,x_2)\in\s_2'(K\cdot(1,\a)).$
    By the definition of the mutation along $K\cdot(1,\a),$ we get that $\s_2'(K\cdot(1,\a))=\s_{4}\Big(K\cdot(1,\a)\otimes K\cdot(1,\a)\Big),$ and since the Kronecker product of $(1,\a)$ and $(1,\a)$ equals $$\begin{bmatrix}
        1 & \a\\
        \a & \a^{2}
    \end{bmatrix},$$
    we conclude that $K\cdot(1,\a)\otimes K\cdot(1,\a)$ is equal to $\{(x,\a x,\a x,\a^{2}x):x\in K\}.$ If $N=\{(x,\a x,\a^{2}x):x\in K\},$ we can recover $K\cdot(1,\a)\otimes K\cdot(1,\a)$ as $\mu\cdot(N\oplus0),$ where $\mu\in\Gl_4(k_0)$ is the matrix $$\begin{bmatrix}
        1 & 0 & 0 & 0\\
        0 & 1 & 0 & 0\\
        0 & 1 & 0 & 1\\
        0 & 0 & 1 & 0
    \end{bmatrix}.$$ 
    Therefore, by equivariance, 
    $$\s_2'(K\cdot(1,\a))=\s_4\Big(K\cdot(1,\a)\otimes K\cdot(1,\a)\Big)=\mu\cdot(\s_3(N)\oplus 0)=\{(a,b,b,c):(a,b,c)\in\s_3(N)\},$$
    so $(0,0,x_1,x_2)=(a,b,b,c)$ for some $(a,b,c)\in\s_3(N).$ This implies that $a=0=x_1=b,$ hence $\e_0:=c=x_2$ has to be non-zero as $(0,0,x_1,x_2)$ was not zero. We have proved that there is some non-zero $\e_0\in M$ such that $(0,0,\e_0)\in\s_3(N)=\s_3(\{(x,\a x, \a^{2}x):x\in K\}).$   
    The identity $$\begin{bmatrix}
        0 & 1 & 1\\
        0 & 0 & 1\\
        1 & 0 & 1
    \end{bmatrix}\cdot\begin{bmatrix}
        \e_0 & 0 & 0\\
        0 & \e_1 & 0\\
        0 & 0 & \e_2
    \end{bmatrix}=\begin{bmatrix}
        0 & \e_1 & \e_2\\
        0 & 0 & \e_2\\
        \e_0 & 0 & \e_2
    \end{bmatrix}$$
    and the fact that the left-most matrix is in $\Gl_3(k_0)$ shows that the $S$-submodule of $\s_3\Big(\{(x,\a x,\a^{2}x):x\in K\}\Big)$ generated by $(0,0,\e_0),$ $(\e_1,0,0)$ and $(\e_2,\e_2,\e_2)$ has length three, which is absurd: the length of $\s_3\Big(\{(x,\a x,\a^{2}x):x\in K\}\Big)$ is $2\cdot\dim_K(K\cdot(1,\a,\a^{2}))=2,$ because $\s$ is a $2$-inflator.
\end{proof}

\subsection{A formula for the Inflator}

The line of argument in Sections 2 and 3 of \cite{dp4} holds verbatim for fields of characteristic 2. In particular, any inflator $\s:\Dir_K(K)\to D_{\bullet}$ satisfying the Weak Assumptions admits a mutation $\s':\Dir_K(K)\to D_{\bullet}'$ which is \emph{isotypical}, meaning that the directory $D_{\bullet}'$ is isomorphic to $\Dir_S(A\oplus A)$ where $A$ is an $S$-module of length 1. 
Indeed, it is enough to mutate along a line in $K^{2}$ whose slope is wild, and wild elements of $K$ exist by Remark \ref{wildex}, cf.~\cite[Corollary 2.3]{dp4}. Since the conditions of the Weak Assumptions are mutation-invariant, the mutation $\s'$ satisfies the Weak Assumptions as well. 
\begin{defin}
    \label{str}
We say that an inflator satisfies the \emph{Strong Assumptions} if it is isotypical and satisfies the Weak Assumptions.    
\end{defin}

The following fact is stated for fields of characteristic different than 2, but we note that its proof is independent of the characteristic of the field. Hence, we decide to state it as a fact that applies to fields of \emph{arbitrary} characteristic, even if this does not correspond to the original statement. Note that malleability of $\s$ intervenes in the proof of this fact.

\begin{fact}[Cf.~{\cite[Theorem 3.15]{dp4}}]
\label{ingr}
    Let $\s$ be a $k_0$-linear 2-inflator on a field $K$ of arbitrary characteristic satisfying the Strong Assumptions. Then there are
    \begin{itemize}[wide]
        \item A field $k$ extending $k_0,$
        
        \item A subring $R\subseteq K,$

        \item An ideal $I\leq R,$

        \item An isomorphism of $k_0$-algebras $R/I\cong k[\e]:=k[\e]/(\e^{2})$ 
    \end{itemize}
    such that $\s$ is isomorphic to 
    \begin{align*}
        \s:\Dir_K(K)&\to\Dir_k(k[\e])\\
        \s_n(V)&=\{(\gres(x_1),\.,\gres(x_n)):(x_1,\.,x_n)\in V\cap R^{n}\},
    \end{align*}
    where $\gres$ is the quotient map $R\twoheadrightarrow R/I\cong k[\e].$ Moreover,
    \begin{itemize}[wide]
        \item $R=R_\s$ is the fundamental ring of $\s,$ $I$ is the fundamental ideal, and $\gres$ is the generalized residue map.

        \item $R$ is a local ring, whose unique maximal ideal is the pullback of $k\cdot\e$ along $\gres$.

        \item $\Frac(R)=K.$ 
    \end{itemize}
\end{fact}

Through such analysis, Johnson finds that if $\a\in K$ specializes to $\mu\in\End_k(k^2)=M_2(k),$ then $\mu$ has a repeated eigenvalue in $k$, cf.~\cite[Corollary 3.13]{dp4}. If $\mu=\begin{bmatrix}
    a & b\\
    c & d\\
\end{bmatrix},$ then its characteristic polynomial $X^2-\tr(\mu)X+\det(\mu)\in k[X]$ has one root $\l\in k$ of algebraic multiplicity 2, i.e.~
$X^2-\tr(\mu)X+\det(\mu)=(X-\l)^2=X^2+\l^2,$
from which we conclude that $-\tr(\mu)=\tr(\mu)=a+d=0$ and $\det(\mu)=ad-bc=\l^{2}.$ It follows that $a=-d=d$ and that $\det(\mu)=a^2-bc=a^2+bc=\l^2.$ In fact, this argument holds for \emph{any} malleable isotypical 2-inflator $\xi$ with no weak $V^2$-type mutations. We highlight this phenomenon as a corollary. 
\begin{cor}
\label{spf}
    Let $\xi:\Dir_K(K)\to\Dir_k(N)$ be a 2-inflator satisfying the Strong Assumptions. If $\a\in R_\xi,$ then it specializes to a matrix $\mu\in M_2(k)$ whose trace is null and whose determinant is a square in $k.$  
\end{cor}

\section{The Associated Valuation}
\label{sec3}

In this section we adopt the notations of Chapter 4 of \cite{dp4}, where
$Q:=\{\a\in R:\gres(\a)=x+0\e\text{ for some }x\in k\}$ is a local subring of $R$ having $I$ as its maximal ideal, and $\p:=\{\a\in R:\gres(\a)=0+x\e\text{ for some }x\in k\}$ is the maximal ideal of $R.$ 

\begin{remark}
\label{qproper}
$Q$ is a proper subring of $R.$ Indeed, if $R=Q,$ then $\p=I$ and thus $R/\p=R/I\cong k[\e]$ is a field, which is absurd.
\end{remark}

\begin{lema}[Cf.~{\cite[Lemma 4.3]{dp4}}]
\label{quad}
    Suppose $\a\in K$ satisfies some monic quadratic equation over $R,$ i.e.~$\a^2+b\a+c=0$ for some $b,c\in R.$ If $\a$ is wild, then $b\in\p$ and $c+\p$ is a square in $R/\p.$ 
\end{lema}
\begin{proof}
    Let $b_0,b_1,c_0,c_1\in k$ be such that $\gres(b)=b_0+b_1\e$ and $\gres(c)=c_0+c_1\e.$
    Since $R/\p\cong k,$ we know that $b\in\p$ if and only if $b_0=0$ and $c+\p$ is a square in $R/\p,$ which in turn is equivalent to $c_0$ being a square in $k.$ Following the proof of \cite[Lemma 4.3]{dp4}, if $\s'$ is the mutation of $\s$ along the line $L=K\cdot(1,\a),$ then 
    $$\s_2'(L)=\s_4\Big(\{(x,\a x,\a x, \a^2 x):x\in K\}\Big)=\{(s,t,t,-b_0t-c_0t):s,t\in k\},$$
    is the graph of the endomorphism $\mu:k^2\to k^2$ for which $\mu(s,t)=(t,-b_0t-c_0s),$ i.e.~of the matrix $$\mu=\begin{bmatrix}
        0 & 1\\
        -c_0 & -b_0
    \end{bmatrix}.$$  
    Since we mutated along $K\cdot(1,\a),$ Fact \ref{tamingwild} tells us that $\a$ is tame with respect to $\s',$ and moreover $\a\in R'$ specializes to $\mu.$ Since mutation preserves malleability and isotopy, and a mutation of $\s'$ is again a mutation of $\s,$ then $\s'$ also satisfies the Strong Assumptions. Hence $b_0=\tr(\mu)=0$ and $c_0=\det(\mu)$ is a square in $k,$ by Corollary \ref{spf}.
\end{proof}

\begin{remark}
    With the hypotheses of Lemma \ref{quad}, we can even choose some $d\in Q$ such that $c-d^2\in\p.$ Indeed, if $\gres(c)=c_0+c_1\e$ and $c_0=d_0^2$ for some $d_0\in k,$ then there is some $d\in Q$ such that $\gres(d)=d_0,$ because $\gres$ is onto. Therefore $\gres(d^2)=d_0^2=c_0,$ so $\gres(c-d^2)=c_1\e\in k\e,$ i.e.~$c-d^2\in\p.$ 
\end{remark}

Recall the following fact. It is true independently of the characteristic of $K.$

\begin{fact}[Cf.~{\cite[Lemma 4.1]{dp4}}]
\label{w2}
    Let $a,b,c\in K.$ Then there exist some $x,y,z\in R$ such that $ax+by+cz=0$ and at least one of $x,y$ and $z$ lies in $R^{\times}.$
\end{fact}

\begin{lema}[Cf.~{\cite[Lemma 4.4]{dp4}}]
\label{degwild}
    Let $\a\in K^{\times}.$ Then $\a$ or $\a^{-1}$ is integral over $R.$ Moreover, one of $\a$ or $\a^{-1}$ satisfies a monic polynomial of degree $d,$ where $$d=\begin{cases}
        1 &\text{ if }\a\text{ is tame,}\\
        2 &\text{ if }\a\text{ is wild.}
    \end{cases}$$ 
\end{lema}
\begin{proof}
    As in the proof of \cite[Lemma 4.4]{dp4}, the result holds if $\a$ is tame. Suppose that $\a$ is wild. By Fact \ref{w2}, there are some $x,y,z\in R$ such that $x+y\a+z\a^2=0,$ and at least one of $x,y$ and $z$ is invertible in $R.$ If $z\in R^{\times}$ or if $x\in R^{\times}$ then $\a$ and $\a^{-1}$ are integral over $R$ respectively, in each case witnessed by a monic quadratic polynomial over $R,$ as wanted. Thus we may assume that $y\in R^{\times}$ and that $x,z\in R\setminus R^{\times}=\p.$ In particular $x+y+z\in R^{\times},$ otherwise $y=(x+y+z)+x+z\in\p,$ which is impossible.
    
    Let $\b=\dfrac{\a}{\a+1},$ so that $\a\b+\b=\a$ and $\a=\dfrac{\b}{\b+1}.$ Then 
    \begin{align*}
        x+y\a+z\a^2&=0\\
        x+y\dfrac{\b}{\b+1}+z\dfrac{\b^2}{\b^2+1}&=0\\
        x(\b^2+1)+y\b(\b+1)+z\b^2&=0\\
        (x+y+z)\b^2+y\b+x&=0\\
        \b^2+\dfrac{y}{x+y+z}\b+\dfrac{x}{x+y+z}&=0.
    \end{align*}
    Since $\b$ is a fractional linear transformation of $\a,$ it is wild as well. By Lemma \ref{quad}, $\dfrac{y}{x+y+z}\equiv 0\text{ mod }\p,$ but $\dfrac{y}{x+y+z}\equiv\dfrac{y}{y}=1\text{ mod }\p,$ which is absurd. 
\end{proof}

\begin{cor}
\label{intclval}
    The integral closure $\O$ of $R$ (in $K$) is a valuation ring of $K,$ and if $\a\in\M\setminus R,$ then there are some $p,q\in R$ such that $\a^{2}+p\a+q=0.$
\end{cor}
\begin{proof}
    We know that $\a$ is wild, so either the conclusion holds or there are some $p,q\in R$ such that $\a^{-2}+p\a^{-1}+q=0,$ meaning that $1+p\a+q\a^{2}=0.$ Then $1=p\a+q\a^{2}\in\M,$ an absurd.
\end{proof}

The exact same proofs of Section 4.3 of \cite{dp4} hold in characteristic 2. The following fact summarizes these results.
\begin{fact}[Cf.~{\cite[Lemma 4.6, Corollary 4.7 and Lemma 4.8]{dp4}}]
\label{wildpos}
Let $\M$ denote the maximal ideal of the valuation ring $\O.$ Then $R\cap \M=\p,$ $\O$ is non-trivial, and if $\a\in\O,$ then $\a$ is tame if and only if $\a\in R.$
\end{fact}

\begin{propo}
\label{resfield}
    The induced map $k\cong R/\p=R/(R\cap\M)\hookrightarrow\O/\M$ is onto, hence an isomorphism. In particular, $\O$ has residue field isomorphic to $k.$
\end{propo}

We denote by $\res:\O\to k$ the residue map. Note that if $\a\in R,$ then $\gres(\a)=x+y\e$ implies that $\res(\a)=x.$

\begin{proof}[Proof of Proposition \ref{resfield}]
    Let $x\in\O.$ We must find some $y\in R$ such that $x+\M=y+\M.$ If $x\in R,$ take $y=x.$ If $x\in\O\setminus R$ then $x$ is wild by Fact \ref{wildpos}. If $x\in\M,$ take $y=0.$ Hence we may assume that $x\in\O^{\times}\setminus R,$ from which, by Lemma \ref{degwild}, one of $x$ or $x^{-1}$ satisfies a monic quadratic equation over $R.$ We split our analysis accordingly.
    \begin{itemize}[wide]
        \item If $x^2+bx+c=0$ for some $b,c\in R,$ then by Lemma \ref{quad} we know that $b\in\p$ and $c+\p$ is a square in $R/\p,$ i.e.~there is some $d\in R$ such that $c+\p=d^2+\p.$ It follows that $c+\M=d^2+\M,$ so $x^2+bx+c=0$ implies that   
        $(x+d)^2\in\M.$ Therefore $x+d\in\M$ and thus $x+\M=d+\M$ as wanted.

        \item If $x^{-2}+bx^{-1}+c=0$ for some $b,c\in R,$ then the latter item shows that $x^{-1}+\M=d+\M$ for some $d\in R.$ Note that $d\in R^{\times},$ because otherwise $d\in\p\subseteq\M$ and $x^{-1}\in\M,$ contradicting the choice of $x\in\O^{\times}.$ It follows that $x+\M=d^{-1}+\M,$ as wanted.
        \qedhere
    \end{itemize}
\end{proof}

\section{Differential Structure and Density}
\label{sec4}

By well known facts about dual numbers, the map $\gres:R\to k[\e]$ induces a non-trivial derivation $\dd_0:R\to k,$ where $k$ is endowed with its natural $R$-module structure. Indeed, $\gres(x)=\res(x)+\dd_0(x)\cdot\e$ uniquely determines $\dd_0:R\to k.$ Under this point of view, $Q$ equals the subring of constants of $\dd_0.$
Recall that for any $a\in\O\setminus Q$ there is some $a^{\dagger}\in Q$ such that $aa^{\dagger}\in R\setminus Q.$ Any such $a^{\dagger}$ is called a \emph{neutralizer} of $a,$ cf.~\cite[Definition 5.5]{dp4}.
Also, in \cite[Definition 5.12]{dp4}, Johnson defines a secondary \say{valuation} map $\val_{\dd}:\O\to\G\cup\{\infty\},$ where $\G$ is the value group of $\O$ and where 
$$\val_{\dd}(a)=\begin{cases}
    \infty &\text{ if }a\in Q,\\
    -\val(a^{\dagger}) &\text{ if }a^{\dagger}\text{ is a neutralizer of }a.
\end{cases}$$
This secondary valuation is well defined: any element $a\in\O\setminus Q$ admits a neutralizer $a^{\dagger}$ and all neutralizers of $a$ have the same valuation, cf.~\cite[Lemma 5.6 and Lemma 5.13]{dp4}. 

\begin{remark}
\label{neutinv}
    Note that if $a\in\O\setminus Q,$ $a'\in\O$ and $q\in Q$ are such that $a=q+a',$ then $\val_{\dd}(a)=\val_{\dd}(a').$ This holds because if $a^{\dagger}$ neutralizes $a,$ then $a^{\dagger}$ neutralizes $a',$ because $a'a^{\dagger}=aa^{\dagger}+qa^{\dagger}\in R\setminus Q.$ 
\end{remark}

\begin{fact}[Cf.~{\cite[Lemma 5.17]{dp4}}]
\label{sqa}
    If $a\in\O$ and $\val(a)+\val_{\dd}(a)>0,$ then $a^{2}\in R.$
\end{fact}
\begin{proof}
    If $a\in Q$ then $a^2\in R$ automatically. If $a\in\O\setminus Q,$ let $a^{\dagger}$ be a neutralizer of $a,$ so that $\val(a)>\val(a^{\dagger}),$ i.e., $a/a^{\dagger}\in\M.$ Also, $a/a^{\dagger}$ is wild, because otherwise $a/a^{\dagger}\in R$ and $a=a^{\dagger}\cdot(a/a^{\dagger})\in I\cdot R\subseteq I\subseteq Q,$ which is absurd. Therefore, by Corollary \ref{intclval}, there are some $p,q\in R$ such that $a^2+paa^{\dagger}+q(a^{\dagger})^{2}=0,$ i.e.~$a^{2}=p(aa^{\dagger})+q(a^{\dagger})^{2}\in R.$ 
\end{proof}

The proof of the following fact is independent of the characteristic of $K.$

\begin{fact}[Cf.~{\cite[Lemma 5.15]{dp4}}]
\label{svaloq}
    Let $a\in\O$ and let $q\in Q.$ Then $$\val_{\dd}(aq)=\begin{cases}
        \infty &\text{ if }\val_{\dd}(a)+\val(q)>0,\\
        \val_{\dd}(a)+\val(q) &\text{ otherwise.}
    \end{cases}$$
\end{fact}

The proofs of Section 5.4 of \cite{dp4} remain true in our context: they do not depend on the characteristic of $K,$ and, in \cite[Lemma 5.20]{dp4}, the assumption that says that no valuative ball is contained in $R$ remains true whenever $\s$ is not weakly of $V^2$-type. We also highlight the following fact, which is independent of the characteristic.

\begin{fact}[Cf.~{\cite[Proposition 5.22]{dp4}}]
\label{QDENSE}
    Let $a\in\O.$ Then for any $\g\in\G$ there are some $q\in Q$ and $a'\in\O$ such that $a=q+a'$ and $v(a')>\g.$
\end{fact}

The following proposition is crucial, and its proof is essentially the same as the one of \cite[Proposition 5.24]{dp4}, but with the necessary changes for our characteristic 2 setting. 

\begin{propo}[Cf.~{\cite[Proposition 5.24]{dp4}}]
\label{o2inq}
    Let $a\in\O.$ Then $a^{2}\in Q.$ 
\end{propo}
\begin{proof}
    Let $u\in I$ be non-zero. Let $\g\in\G$ be such that $-\g<\val_{\dd}(a)$ and $\val(u)<\g.$ Use Fact \ref{QDENSE} to find some $q\in Q$ and some $a'\in\O$ such that $a=q+a'$ and $\val(a')>3\g.$ By Remark \ref{neutinv}, $-\g<\val_{\dd}(a)=\val_{\dd}(a').$ Also, 
    \begin{equation}
    \label{up}
    \val(a'/u)=\val(a')-\val(u)>3\g-\g=2\g>0    
    \end{equation} and, by Fact \ref{svaloq}, 
    \begin{equation}
    \label{5.15}
        \val_{\dd}(a')=\val_{\dd}\left(\dfrac{a'}{u}\cdot u\right)=\val_{\dd}\left(\dfrac{a'}{u}\right)+\val(u)
    \end{equation} unless $\val_{\dd}\left(\dfrac{a'}{u}\right)+\val(u)>0,$ in which case $a'\in Q,$ yielding $a=q+a'\in Q$ and $a^{2}\in Q.$ If $\val_{\dd}\left(\dfrac{a'}{u}\right)+\val(u)\leq 0,$ then equation \ref{5.15} holds, implying that 
    $$-\g<\val_{\dd}(a')=\val_{\dd}\left(\dfrac{a'}{u}\right)+\val(u)<\val_{\dd}\left(\dfrac{a'}{u}\right)+\g,$$
    so that 
    \begin{equation}
        \label{down}
        \val_{\dd}\left(\dfrac{a'}{u}\right)>-2\g.
    \end{equation}
    Adding equation \ref{up} with equation \ref{down} we get that $\val(a'/u)+\val_{\dd}(a'/u)>0.$ Then Fact \ref{sqa} implies that $(a'/u)^{2}\in R,$ so that $(a')^2\in R\cdot u^{2}\subseteq R\cdot I\subseteq I\subseteq Q.$ Finally, $a^{2}=q^2+(a')^2\in Q$ as wanted.
\end{proof}

\begin{cor}[Cf.~{\cite[Proposition 5.32]{dp4}}]
\label{perf}
    $K$ is imperfect.
\end{cor}
\begin{proof}
    If $K$ is perfect, then $\O$ is perfect as well, i.e.~$\O^{2}=\O.$ By Proposition \ref{o2inq}, we get that $\O^{2}\subseteq Q,$ so $\O^{2}\subseteq Q\subsetneq R\subseteq\O,$ which is absurd. 
\end{proof}

We finish this section by summarizing Propositions 5.23, 5.30 and 5.31 of \cite{dp4}, whose proofs are independent of the characteristic. As a reminder, given a valued field $(K,\O),$ a \emph{mock $K/\M$} is a couple $(D,\phi)$ where $D$ is a divisible $\O$-module such that $x\in\O\cdot y$ or $y\in\O\cdot x$ for all $x,y\in D,$ and $\phi:k\to D$ is an $\O$-module injection, whose image is denoted as $D_0.$

\begin{fact}[Cf.~{\cite[Propositions 5.23, 5.30 and 5.31]{dp4}}]
\label{mkm}
    There is an $\O$-module structure extending the $Q$-module structure of $\O/Q,$ making it into a mock $K/\M.$
\end{fact}

\section{Fields of dp-rank 2}
\label{sec5}
Following Johnson's notations, an unstable dp-finite field is called \emph{of valuation type} if its canonical topology is a $V$-topology.  The proof of the following fundamental theorem \cite[Theorem 6.13]{dp4} is independent of the characteristic. We include a proof in order to stress how dp-rank and inflators fit together. 

\begin{fact}[Cf.~{\cite[Theorem 6.13]{dp4}}]
\label{2inf}
    Let $\KK$ be a sufficiently saturated model of (the theory of) an unstable field of characteristic 2 and dp-rank at most 2. If $\KK$ is not of valuation type, then there is a 2-inflator $\s$ defined on $\Dir_\KK(\KK)$ satisfying the Strong Assumptions.
\end{fact}
\begin{proof}
Let $k_0\preceq\KK$ be a magic subfield. If $\Lambda$ is the lattice of type-definable $k_0$-linear subspaces of $\KK,$ then $\Lambda$ admits a non-zero pedestal $G,$ inducing a malleable $m$-inflator $\s$ with fundamental ideal $R_\s=\Stab(G)=\{x\in\KK:xG\subseteq G\},$ where $m\in\{1,2\}.$ 
By \cite[Proposition 5.32]{dp3}, the fundamental ring of a malleable 1-inflator is a valuation ring, so $\KK$ would be of valuation type if $m=1,$ cf.~\cite[Corollary 6.12]{dp4}.  
If $m=2,$ by the same Corollary, if some mutation of $\s$ is weakly of multi-valuation type, then $\KK$ is of valuation type. Therefore no mutation of $\s$ is of weak multi-valuation type, and, in particular, no mutation of $\s$ is of weak $V^2$-type. By \cite[Corollary 2.3]{dp4}, $\s$ admits an isotypical mutation, which will satisfy the Weak Assumptions too. Hence, such a mutation satisfies the Strong Assumptions, as wanted.
\end{proof}

The following lemma is folklore. We include a proof for the sake of completeness.

\begin{lema}
\label{fdpperf}
    Fields of finite dp-rank are perfect.
\end{lema}
\begin{proof}
    Let $K$ be a field of finite dp-rank and characteristic $p>0.$ Then $K\cong K^{p}$ via Frobenius, which is a definable isomorphism. Thus we have that $\dprk(K)=\dprk(K^{p}).$ If $e=[K:K^{p}]$ is finite, then $K\cong (K^{p})^{e}$ as $K^{p}$-vector spaces, and this isomorphism is also definable in the field structure. By additivity of $\dprk,$ $$\dprk(K^{p})=\dprk(K)=\dprk((K^{p})^{e})=e\cdot\dprk(K^{p}),$$ so either $e=1,$ i.e.~$K=K^{p},$ or $\dprk(K^{p})=0,$ in which case $K^{p}$ and $K$ are both finite, hence equal. Now we have to rule out the possibility of $e$ being infinite. Suppose that this is the case, let $n<\omega$ and let $a_1,\.,a_n\in K$ be $K^{p}$-linearly independent. Then the map $(K^{p})^{n}\to K$ given by $(b_1,\.,b_n)\mapsto\sum_{i=1}^{n}b_i\cdot a_i$ is a definable injective homomorphism of $K^{p}$-vector spaces, implying by additivity that $n\cdot\dprk(K^{p})=\dprk((K^{p})^{n})\leq \dprk(K).$ Since $\dprk(K)$ is finite, we reach a contradiction by choosing $n<\omega$ such that $n\cdot\dprk(K^{p})>\dprk(K).$   
\end{proof}

\begin{propo}[Cf.~{\cite[Proposition 6.28]{dp4}}]
\label{char0}
    Let $\KK$ be a monster model of an unstable field with $\dprk(\KK)\leq 2.$ If $\KK$ is not of valuation type, then $\char(\KK)=0.$ 
\end{propo}
\begin{proof}
    By \cite[Proposition 6.28]{dp4}, $\char(\KK)$ is either 0 or 2. It cannot be 2, because in this case, by Fact \ref{2inf}, there is a 2-inflator on $\KK$ satisfying the Strong Assumptions. By Corollary \ref{perf}, $\KK$ is imperfect. But then $\KK$ cannot have finite dp-rank, by Lemma \ref{fdpperf}.
\end{proof}

\begin{teorema}[Cf.~{\cite[Theorem 6.29]{dp4}}]
\label{teo1}
    If $K$ is a field of positive characteristic and of dp-rank at most 2, then $K$ is either stable or of valuation type.
\end{teorema}
\begin{proof}
    The result holds if the characteristic of $K$ is odd by \cite[Theorem 6.29]{dp4}. If $\char(K)=2,$ then $K$ is of valuation type by Proposition \ref{char0}. Indeed, if $\KK$ is a sufficiently saturated extension of $K,$ then the canonical topology on $\KK$ is a definable $V$-topology, whose restriction to $K$ corresponds to the canonical topology of $K,$ which, by definability, is again a $V$-topology.  
\end{proof}

\section{2-inflators and \texorpdfstring{$W_2$}{}-topologies in Positive Characteristic}
\label{sec6}

All fields and rings mentioned in this section are assumed to be of positive characteristic, unless stated otherwise.

\subsection{Pre-Diffeo-Valued Fields}

Given a valued field $(K,\O)$, if $(D,\phi)$ is a mock $K/\M$, then \cite[Corollary 8.4]{dp4} shows that there is a secondary valuation map $\val:D\to\G\cup\{\infty\}$ such that 
\begin{enumerate}[wide]
\item $\val(x)\leq 0$ or $\val(x)=\infty$ for any $x\in D,$
            
\label{1}\item $\val(x)=\infty$ if and only if $x=0,$

\label{2}\item $\val(x)\geq0$ if and only if $x\in\img(\phi),$
    
\label{3}\item For any $a\in\O$ and any $x\in D,$ $$\val(ax)=\begin{cases}\infty &\text{ if }v(a)+\val(x)>0,\\
v(a)+\val(x) &\text{ if }v(a)+\val(x)\leq 0,
\end{cases}$$
            
\label{4}\item $\val(x+y)\geq\min\{\val(x),\val(y)\}$ for any $x,y\in D,$ 
            
\label{5}\item $y\in\O\cdot x$ whenever $\val(y)\geq\val(x),$ and $y\in\M\cdot x$ whenever $\val(y)>\val(x).$

\label{6}\item For any $\g\in\G$ there is some $x\in D$ such that $\val(x)\leq\g.$

\label{7}
\end{enumerate}   
These properties hold because, following \cite[Proposition 8.2]{dp4}, for any sufficiently resplendent elementary extension $(K^{*},\O^{*},D^{*})$ of $(K,\O,D),$ there is an isomorphism of $\O^{*}$-modules $D^{*}\cong K^{*}/\M^{*},$ on which we can define 
\begin{equation}
\label{valet}
    \val^{*}(x+\M^{*})=\begin{cases}
    \infty & \text{ if }x\in\M^{*},\\
    v^{*}(x) & \text{ if }x\not\in\M^{*}.
\end{cases}
\end{equation}
If we expand the language with a function symbol $\val,$ then $(K^{*},\O^{*},D^{*},\val^{*})$ satisfies the following first-order properties:
\begin{itemize}[wide]
            \item For every $x\in D\setminus\{0\}$ there is some $a\in\O$ such that $ax\in D_0\setminus\{0\},$
            \item For any $a,b\in\O,$ if $ax,bx\in D_0\setminus\{0\},$ then $v(a)=v(b),$
            \item For all $x\in D,$ $$\val(x)=\begin{cases}\infty &\text{ if }x=0,\\-v(a) &\text{ if }ax\in D_0\setminus\{0\},\end{cases}$$
\end{itemize}
out of which Properties 1 to 7 follow.        
Therefore, by resplendency, there is a realization of $\val$ in $(K,\O,D)$ satisfying the same properties. Note that the second part of Property \ref{6} is not in the original list of properties of \cite[Corollary 8.4]{dp4}, so it needs to be checked. Consider a sufficiently resplendent elementary extension $(K^{*},\O^{*},D^{*},\val^{*})$ of $(K,\O,D,\val),$ and assume that $D^{*}=K^{*}/\M^{*}$ and that $\val^{*}$ is defined as in equation \ref{valet}. Then, if $\val^{*}(x+\M^{*})<\val^{*}(y+\M)$ we have that $\val^{*}(x+\M^{*})$ cannot be $\infty,$ so $\val^{*}(x+\M^{*})=v^{*}(x)$ and $x\not\in\M^{*}$ i.e.~$1/x\in\O^{*}.$ If $y\in\M^{*},$ then $y/x\in\M^{*}$ and $y=(y/x)\cdot x,$ so the conclusion follows. If $y\not\in\M^{*},$ then $\val^{*}(y+\M^{*})=v^{*}(y)>v^{*}(x),$ so $y/x\in\M^{*}$ and again the conclusion follows. Therefore the sentence 
$\A x,y\in D\,\big(\val(y)>\val(x)\to\E a\in\M\,(y=ax)\big)$ 
holds in $(K^{*},\O^{*},D^{*},\val^{*})$ and thus in $(K,\O,D,\val).$ 

\begin{defin}
\label{pddvf}
A \emph{pre-diffeo-valued field} is the data of:
\begin{enumerate}
    \item A valued field $(K,\O),$

    \item A mock $K/\M$ $(D,\phi),$

    \item A subring $R\subseteq\O$ and a surjective derivation $\dd:R\to k.$ If we denote $Q=\ker\dd,$ then $\O,R,D$ and $k$ are $Q$-modules in a natural way.
    
    \item A surjective $Q$-module morphism $\pi:\O\to D$ such that $\ker\pi=Q,$    
\end{enumerate}
 such that the diagram   
\begin{center}
    \begin{tikzcd}
\O \arrow[rr, "\pi"]                 &  & D                     \\
                                     &  &                       \\
R \arrow[uu, hook] \arrow[rr, "\dd"] &  & k \arrow[uu, "\phi"']
\end{tikzcd}
\end{center}
commutes. If $\char(K)$ is odd, we set $\pi$ to be a derivation, and if $\char(K)=2,$ we ask that $\O^{2}=\{a^{2}:a\in\O\}\subseteq Q.$ 
\end{defin}

 Here we endow $k$ with its natural $\O$-module structure, so that Leibniz Rule for $\dd$ in $k$ makes sense: if $r\in R$ and $x\in k,$ then $r\cdot x=\res(r)x.$ Note that pre-diffeo-valued fields are first-order axiomatizable, and their theory is consistent: if $(K,D,\d)$ is a diffeo-valued field, then $\pi=\d,$ $R=\{x\in\O:\val(\d(x))\geq 0\}$ and $\dd=\d|_{R}$ make $(K,D,\d)$ into a pre-diffeo-valued field. More precisely, if $\phi:k\to D$ is the $\O$-module injection, then $\img(\phi)=\img(\d|_R),$ so $\dd:=\phi^{-1}\circ\d|_R:R\to k$ is a well defined surjective derivation with kernel $\ker(\d)\cap R=\ker(\d)=\{x\in\O:\val(\d(x))>0\}.$
 
\begin{lema}
\label{ryq}
    The map $v_\pi:\O\to\G\cup\{\infty\}$ defined by $v_\pi(x)=\val(\pi(x))$ satisfies that $R=\{x\in\O:v_\pi(x)\geq0\}$ and $Q=\{x\in\O:v_\pi(x)>0\}.$  
\end{lema}
\begin{proof}
    First, if $x\in R,$ then $\dd(x)\in k$ and $\pi(x)\phi(\dd(x))\in\img(\phi).$ By Property \ref{3}, $v_\pi(x)=\val(\pi(x))\geq0.$ Also, if $v_\pi(x)\geq0$ for some $x\in\O,$ then, by the same property, $\pi(x)\in\img(\phi),$ so there is some $z\in k$ with $\pi(x)=\phi(z).$ As $\dd$ is onto, there is some $y\in R$ such that $\dd(y)=z.$ It follows that $\pi(x)=\phi(\dd(y))=\pi(y),$ so $x-y\in Q\subseteq R.$ As $y\in R,$ it follows that $x\in R$ too. Second, if $x\in\O,$ then Properties \ref{1} and \ref{2} yield that
    $$x\in Q\Sii\pi(x)=0\Sii\val(\pi(x))=\infty\Sii\val(\pi(x))>0,$$ as wanted.
\end{proof}

In this way, if $\char(K)$ is odd, pre-diffeo-valuation data coincides with the diffeo-valuation data of Section 8.4 of \cite{dp4}.

\begin{lema}
\label{k0alg}
    Let $k_0=\F_p^{alg}\cap K,$ where $p=\char(K).$ Then $Q$ is a $k_0$-algebra.
\end{lema}
\begin{proof}
    Let $q\in k_0,$ so that $q$ is algebraic over $\F_p.$ There is some $n\geq1$ such that $q^{p^{n}}=q.$ Therefore $v(q)=v(q^{p^{n}})=p^{n}v(q),$ so $v(x)$ is either $0$ or $\infty,$ because $\G$ is torsionless. Therefore $k_0\subseteq\O,$ so $\O$ is a $k_0$-algebra. If $p$ is odd, $\pi$ is a derivation and $\pi(q)=\pi(q^{p^n})=0,$ i.e.~$q\in Q.$ If $p=2,$ then $q=q^{p^{n}}=q^{2^{n}}\in\O^{2}\subseteq Q$ as well.  
\end{proof}

\begin{defin}
We say that a pre-diffeo-valued field as in Definition \ref{pddvf} is \emph{dense} if $Q$ is dense in $\O$ with respect to the valuation topology, i.e.~for any $a\in\O$ and any $\g\in\G$ there is some $q\in Q$ such that $a=q+r$ and $v(r)>\g.$ We also say that the pre-diffeo-valuation data is dense. 
\end{defin}

This property is also first-order axiomatizable. As above, any dense diffeo-valued field is again a dense pre-diffeo-valued field.

\begin{remark}[Cf.~{\cite[Remark 8.13]{dp4}}]
\label{fibdens}
    If the pre-diffeo-valuation data is dense, then all fibers of $\dd$ are dense in $\O,$ and $\O$ is not trivial.
\end{remark}
\begin{proof}
First, we know that $Q,$ which is the fiber of $0$ under $\dd,$ is dense in $\O.$ If $y\in k,$ then there is some $r\in R$ such that $\dd(r)=y$ because $\dd$ is surjective. Then the fiber of $y$ under $\dd$ is exactly $r+ Q.$ Indeed, $\dd(r+q)=\dd(r)=y$ for any $q\in Q$ and if $s\in R$ is such that $\dd(s)=\dd(r)=y,$ then $s\in r+Q$ for $s-r\in\ker(\dd)=Q.$ Now, if $Q$ is dense in $\O,$ then $r+Q$ is dense in $r+\O=\O.$
Second, if $\O=K,$ then $\M=\{0\}$ and the valuation topology is discrete. If $Q$ is dense in $\O$ with respect to the discrete topology, then $Q=R=\O=K,$ implying that $\dd:R\to k$ is trivial, hence not surjective.
\end{proof}

For the rest of this section, let $(K,\O,D,R,\phi,\dd,\pi)$ be a fix dense pre-diffeo-valued field.
 
\begin{lema}[Cf.~{\cite[Lemma 8.14]{dp4}}]
\begin{enumerate}[wide]
\label{top}
    \item $R$ and $Q$ are proper subrings of $K.$

    \item $I:=Q\cap\M$ is a proper ideal of $R$ and $Q.$

    \item $\Frac(Q)=\Frac(R)=K.$

    \item $Q$ is a local ring with maximal ideal $I.$

    \item $I\neq 0.$
\end{enumerate}
\end{lema}
\begin{proof}
    \begin{enumerate}[wide]
        \item $R$ is a subring of $\O$ by definition and $Q$ is a subring of $R$ because $\dd(ab)=a\cdot\dd(b)+b\cdot\dd(a)=0$ whenever $\dd(a)=\dd(b)=0,$ i.e.~whenever $a,b\in Q.$ They are proper subrings as $\O$ is proper whenever the pre-diffeo-valuation data is dense.

        \item $I$ is an ideal of $Q$ because it is an ideal of $\O$ intersected with $Q.$ It is proper because $1\not\in I.$ We have to prove that $I$ is an ideal of $R.$ To this end, let $i\in I$ and $a\in R.$ Then $i\in\M$ and $a\in\O,$ so $ia\in\M.$ Also $v_\pi(a)\geq 0$ and, by Property \ref{4}, $\val(\pi(ia))=\val(i\cdot\pi(a))\geq v(i)+v_\pi(a)\geq v(i)>0,$ so $ia\in Q.$ 

        \item It is enough to show that $\O\subseteq\Frac(Q).$ Let $a\in\O$ and let $q\in Q$ and $r\in\M$ be such that $0=q+r$ and $v(r)+v_\pi(a)>0.$ Then $r=-q\in Q$ and, by Property \ref{4}, we have that $v_\pi(ra)=\val(\pi(ra))=\val(r\cdot\pi(a))=\infty,$ i.e.~$ra\in Q.$   

        \item We need to show that $Q\setminus I\subseteq Q^{\times}.$ Indeed, if $q\in Q$ is not in $I,$ then $\dd(q)=0$ and $v(q)=0.$ Therefore $v(q^{-1})=0$ and $\dd(q^{-1})=-q^{-2}\cdot\dd(q)=0,$ i.e.~$q^{-1}\in Q.$  

        \item If $I=0$ then $Q$ is a field, for $I$ is maximal. Then $Q=\Frac(Q)=K,$ contradicting properness.
        \qedhere
    \end{enumerate}
\end{proof}

\begin{lema}[Cf.~{\cite[Lemma 8.22]{dp4}}]
\label{reshat}
Let $\gres:R\to k^{2}$ be the map defined by $\gres(a)=(\res(a),\dd(a)).$ 
    \begin{enumerate}[wide]
        \item The map $\res$ induces a ring isomorphism between the quotient $Q/I$ and $k,$ so $k$ is a simple $Q$-module.

        \item $R$ and $I$ are $Q$-submodules of $K.$ 

        \item The map $\gres:R\to k^{2}$ induces an isomorphism of $Q$-modules between $R/I$ and $k^{2}.$ Therefore $R/I$ is a semi-simple $Q$-module of length 2.
    \end{enumerate}
\end{lema}
\begin{proof}
    \begin{enumerate}[wide]
        \item We have to check that $\ker(\res)\cap Q=I$ and that $\res|_Q$ is onto. The first follows from $\ker(\res)=\M,$ and the second because of density: if $x\in k$ and $r\in R$ is such that $\dd(r)=x,$ then one can write $r=q+r'$ for some $q\in Q$ and $v(r')>0.$ Then $\res(r)=\res(q).$   

        \item $R$ is an overring of $Q$ and $I$ is an ideal of $Q.$

        \item We have to check that $\gres$ is a $Q$-module morphism, $\ker(\gres)=I$ and that $\gres$ is onto. First, if $q\in Q$ and $r\in R,$ then 
        \begin{align*}
            \gres(qr)&=(\res(qr),\dd(qr))\\
            &=(\res(q)\res(r),q\cdot\dd(r))\\
            &=(\res(q)\res(r),\res(q)\dd(r))\\
            &=\res(q)(\res(r),\dd(r))\\
            &=\res(q)\gres(r)\\
            &=q\cdot\gres(r).
        \end{align*}
        We also have that $\gres(r)=(\res(r),\dd(r))=(0,0)$ if and only if $r\in\M\cap\ker(\dd)=\M\cap Q=I.$ Finally, let $(x+\M,y+\M)\in k^{2},$ with $x,y\in\O.$ Use Remark \ref{fibdens} to find some $r\in R$ such that $\dd(r)=y+\M$ and $v(r-x)>0.$ Then $\gres(r)=(r+\M,\dd(r))=(x+\M,y+\M)$ as wanted.  
        \qedhere
    \end{enumerate}
\end{proof}

\begin{lema}[Cf.~{\cite[Lemma 8.23]{dp4}}]
\label{qw2}
    Let $a,b,c\in K.$ Then the $Q$-submodule generated by $\{a,b,c\}$ is generated by a two-element subset of $\{a,b,c\}.$ 
\end{lema}
\begin{proof}
    Without loss, we may assume that $v(c)\leq\min\{v(a),v(b)\},$ so that $a/c$ and $b/c$ are in $\O.$ Also, we may assume that $v_\pi(a/c)\leq v_\pi(b/c),$ i.e.~$\val(\pi(a/c))\leq\val(\pi(b/c)).$ If $v_\pi(b/c)=\infty,$ then $b/c\in Q$ and $b\in Qc,$ as wanted. If not, then $\val(\pi(a/c))$ is finite. By Property \ref{6}, there is some $x\in\O$ such that $x\cdot\pi(a/c)=\pi(b/c).$ Let $q\in Q$ and $r\in\M$ be such that $x=q+r$ and $v(r)+v_{\pi}(a/c)>0.$ Then $$\pi(b/c)=x\cdot\pi(a/c)=q\cdot\pi(a/c)+r\cdot\pi(a/c)=q\cdot\pi(a/c)=\pi(qa/c),$$
    so that $b/c\in q\cdot(a/c)+Q\subseteq Q\cdot(a/c)+Q$ and $b\in Qa+Qc,$ as wanted.
\end{proof}

\begin{lema}
\label{valinv}
    Let $a\in\O^{\times}.$ Then $v_\pi(a)=v_\pi(a^{-1}).$ 
\end{lema}
\begin{proof}
    If $char(K)$ is odd, then $\pi$ is a derivation and $\pi(a^{-1})=-a^{-2}\cdot\pi(a),$ so Property \ref{4} shows that
    \begin{align*}
        \val(\pi(a^{-1}))&=\val(-a^{2}\cdot\pi(a))\\
        &=\begin{cases}
        \infty & \text{ if }\val(\pi(a))>0,\\
        \val(\pi(a)) & \text{ otherwise.}
    \end{cases}
    \end{align*}
    Thus, if $\val(\pi(a))>0$ then $\val(\pi(a))=\infty=\val(\pi(a^{-1})),$ and if $\val(\pi(a))\leq0,$ we also get $\val(\pi(a^{-1}))=\val(\pi(a))$ as wanted.
    If $\char(K)=2$ and the conclusion does not hold, without loss, $v_\pi(a)<v_\pi(a^{-1}).$ In this case $v_\pi(a)$ is finite. By Property \ref{6}, there is some $x\in\M$ such that $\pi(a^{-1})=x\cdot\pi(a).$ Find some $i\in Q, r\in\M$ such that $x=i+r$ and $v(r)+v_\pi(a)>0.$ Then $i=x-r\in I$ and
    $$\pi(a^{-1})=x\cdot\pi(a)=i\cdot\pi(a)+r\cdot\pi(a)=i\cdot\pi(a)=\pi(ia),$$
    i.e.~$a^{-1}-ia\in Q.$ Then $1-ia^{2}\in Qa.$ Since $\O^{2}\subseteq Q$ and $i\in I,$ then $ia^2\in I$ and $1-ia^2\in Q^{\times}.$ It follows that $1\in Qa,$ i.e.~$a^{-1}\in Q.$ Moreover, $v(a^{-1})=0$ implies that $a^{-1}\in Q^{\times}$ because the unique maximal ideal of $Q$ is $I=Q\cap\M.$ Then $a\in Q,$ meaning that $v_\pi(a)=\infty=v_\pi(a^{-1}),$ a contradiction.  
\end{proof}

\subsection{Diffeo-valuation Inflators}

Recall that if $\Lambda$ is a modular lattice, then the \emph{cube rank} of $\Lambda$ is the maximum $n\in\N$ such that there is a strict $n$-cube in $\L,$ and $\infty$ if no such maximum exists.   

\begin{cor}
\label{rrk}
    The cube rank of $\Sub_Q(K)$ is $2.$ 
\end{cor}
\begin{proof}
    Suppose there is a strict 3-cube in $\Sub_Q(K).$ Then there are four $Q$-submodules $A,B_1,B_2,B_3$ of $K$ such that $A\subseteq B_1\cap B_2\cap B_3$ and $\{B_1,B_2,B_3\}$ are lattice-independent over $A.$ Then none of the $B_i$ is equal to $A,$ so we can choose some $b_i\in B_i/A$ for each $i.$ The $Q$-submodule of $K$ generated by $\{b_1,b_2,b_3\}$ is generated, say, by $\{b_1,b_2\}$ by Lemma \ref{qw2}, i.e.~$b_3\in Qb_1+Qb_2.$ But then $B_3\cap (B_1+B_2)\supseteq Qb_3\supsetneq A,$ contradicting independence over $A$. Therefore the cube rank of $\Sub_Q(K)$ is at most $2.$ 
    Also, by Statement 3 of Lemma \ref{reshat}, $R/I\cong k^{2}$ as $Q$-modules via $\gres,$ so we may pull back the strict 2-cube 
    \begin{center}
        \begin{tikzcd}
                                                  & k^{2} \arrow[rd, no head] &           \\
0\oplus k \arrow[rd, no head] \arrow[ru, no head] &                           & k\oplus 0 \\
                                                  & 0 \arrow[ru, no head]     &          
\end{tikzcd}
    \end{center}
    to a strict 2-cube
    \begin{center}
        \begin{tikzcd}
                                                              & R \arrow[rd, no head] &                       \\
\gres^{-1}(0\oplus k) \arrow[rd, no head] \arrow[ru, no head] &                       & \gres^{-1}(k\oplus 0) \\
                                                              & I \arrow[ru, no head] &                      
\end{tikzcd}
    \end{center}
    implying that the cube rank of $\Sub_Q(K)$ is at least 2.
\end{proof}

\begin{teorema}[Cf.~{\cite[Theorem 8.24]{dp4}}]
\label{dpdvd}
    Let $(K,\O,D,R,\phi,\dd,\pi)$ be a dense pre-diffeo-valued field. As before, put $k_0=\F_p^{alg}\cap K.$ Then there is a malleable $k_0$-linear 2-inflator $\s:\Dir_K(K)\to\Dir_k(k^{2})$ given by $$\s(V)=\{(\widehat{\res}(x_1),\.,\widehat{\res}(x_n)):(x_1,\.,x_n)\in V\cap R^{n}\},$$
    where $V\in\Sub_K(K^{n})$ and $\widehat{\res}$ is as in Lemma \ref{reshat}.
\end{teorema}
\begin{proof}
    Let $\Ca$ be the category of $Q$-modules, and let $F:K\Vect\to\Ca$ be the forgetful functor. Let $G:\Ca\to k_0\Vect$ the forgetful functor, which is well defined by Lemma \ref{k0alg}. By Corollary \ref{rrk}, $F(K)=K$ has cube rank 2. By item 3 of \ref{reshat}, $I$ is a pedestal of $\Sub_Q(K).$ Propositions 8.9 and 8.12 of \cite{dp3} show that there is a malleable 2-inflator $\s:\Dir_K(K)\to\Dir_Q(R/I)$ given by $$\s_n(V)=(V\cap R^{n}+I^{n})/I^{n}\cong\rho^{\oplus n}(V\cap R^{n}),$$ where $\rho:R\to R/I$ is the natural projection and $\rho^{\oplus n}:R^{n}\to(R/I)^{n}$ is coordinate-wise $\rho.$ Under the isomorphism $\Dir_Q(R/I)\cong\Dir_Q(k^{2})\cong\Dir_k(k^2)$ induced by $\gres,$ we get that $$\s(V)=\{(\gres(x_1),\.,\gres(x_n)):(x_1,\.,x_n)\in V\cap R^n\}$$
    for any $V\in\Sub_K(K^{n}),$ as wanted.
\end{proof}

\begin{defin}
    We say that a 2-inflator on a field of positive characteristic is a \emph{diffeo-valuation} inflator if it is isomorphic to the inflator of dense pre-diffeo-valuation data given by Theorem \ref{dpdvd}.
\end{defin}
Since pre-diffeo-valuation data and diffeo-valuation data coincide in fields of odd characteristic, this definition corresponds to the one given by Johnson for odd characteristic fields.

\begin{teorema}[Cf.~{\cite[Theorem 8.26]{dp4}}]
\label{dvi}
    Let $\s$ be an isotypic, malleable 2-inflator on a field of positive characteristic. If no mutation of $\s$ is weakly of $V^2$-type, then $\s$ is a diffeo-valuation inflator.
\end{teorema}
\begin{proof}
    If $\char(K)$ is odd, then \cite[Theorem 8.26]{dp4} shows that there is diffeo-valuation data inducing our pre-diffeo-valuation data, and whose associated inflator is isomorphic to $\s.$ 
    If $\char(K)=2,$ then Fact \ref{ingr} shows that there is a field $k$ extending $k_0$ and an isomorphism of $k_0$-algebras $R/I\cong k[\e],$ where $R$ and $I$ are the fundamental ring and ideal of $\s$ respectively. In fact, such isomorphism is induced by the generalized residue map $\gres:R\to k[\e],$ which is surjective. By Corollary \ref{intclval}, the integral closure $\O$ of $R$ is a valuation ring of $K$, and by Proposition \ref{resfield} the residue field of $\O$ is $k.$ Since $\gres$ is a surjective ring homomorphism, there is a surjective derivation $\dd:R\to k$ such that $\gres(r)=\res(r)+\dd(r)\cdot\e$ for all $r\in R.$ We defined $Q$ to be the kernel of $\dd,$ and proved that $\O^{2}\subseteq Q$ in Proposition \ref{o2inq}. By Fact \ref{mkm}, one can endow $\O/Q$ with an $\O$-module structure that extends its $Q$-module structure and making it a mock $K/\M.$ Therefore the projection $\pi:\O\to\O/Q$ is a $Q$-module surjection with kernel $Q.$ We have recovered all of the pre-diffeo-valuation data. Since $R/I\cong k[\e]\cong k^{2}$ as $k$-modules, the inflator from Theorem \ref{dpdvd} is isomorphic to $\s$ by Fact \ref{ingr}.  
\end{proof}

\begin{cor}[Cf.~{\cite[Corollary 8.27]{dp4}}]
    Let $\s$ be a malleable 2-inflator on a field of positive characteristic. Then some mutation of $\s$ is weakly of $V^2$-type, or a diffeo-valuation inflator.
\end{cor}
\begin{proof}
    If no mutation of $\s$ is of weak $V^2$-type, then there is a mutation $\s'$ of $\s$ which is isotypic. Such a $\s'$ is still malleable and admits no weak $V^2$-type mutations too. Therefore $\s'$ is a diffeo-valuation inflator by Theorem \ref{dvi}.
\end{proof}

\begin{defin}
    Let $S$ be a multi-valuation ring of a field $K$ of characteristic $p>0.$ Recall the notation $k_0=\F_p^{alg}\cap K,$ and let $n\geq 1$ be the number of maximal ideals of $S.$ We say that $S$ is \emph{good} if $|k_0|\geq n.$
\end{defin}

\begin{lema}
\label{k0mult}
    Let $S$ be a good multi-valuation ring of $K.$ If $b\in K\setminus S,$ then there is some $q\in k_0$ such that $1/(b-q)\in S.$  
\end{lema}
\begin{proof}
    If $\O$ is the localization of $S$ on some maximal ideal, then $\O$ is a valuation ring of $K,$ which we know is a $k_0$-algebra by the proof of Lemma \ref{k0alg}. Let $q_1,\.,q_n$ be pairwise different elements of $k_0.$ Then one of $$b,\dfrac{1}{b-q_1},\.,\dfrac{1}{b-q_n}$$ is an element of $S,$ by \cite[Lemma 5.24]{dp3}. Since $b\not\in S,$ the result follows.    
\end{proof}

\begin{lema}[Cf.~{\cite[Lemma 8.29]{dp4}}]
\label{nogood}
    Let $K$ be a dense pre-diffeo-valued field and let $S$ be a good multi-valuation ring of $K$ such that $aS\subseteq Q$ for some $a\in K.$ Then $a=0.$
\end{lema}
\begin{proof}
    If $a\neq0,$ then $v(a)$ is finite. Note that $a=a\cdot1\in aS\subseteq Q\subseteq\O.$ Use Property \ref{7} and surjectivity of $\pi$ to find some $b'\in\O$ such that $\val(\pi(b'))<\min\{\val(\pi(a)),0,-v(a)\}.$ Let $q'\in Q$ and $b\in\M$ be such that $b'=q'+b,$ so that $\val(\pi(b))=\val(\pi(b')).$ There are two cases:
    \begin{enumerate}[wide]
        \item If $b\in S,$ then $ab\in Q$ and 
        \begin{align*}
         0&<\val(\pi(ab))&\text{ (because }ab\in Q)\\
         &=\val(a\cdot\pi(b))&\text{ (because }a\in Q)\\
         &=v(a)+\val(\pi(b))&\text{ (because }v(a)+\val(\pi(b))\leq0)\\
         &\leq 0,
        \end{align*}
        which is absurd.
        
        \item If $b\not\in S,$ given that $S$ is good, use Lemma \ref{k0mult} to find some $q\in k_0$ such that $1/(b-q)\in S.$ By Lemma \ref{k0alg}, such $q$ is also in $Q,$ and by hypothesis, $a/(b-q)\in Q.$
        We claim that $q\in Q$ has to be non-zero. Otherwise $a/(b-q)=a/b\in Q.$ Since $b\in\M,$ we have that $v(a)+\val(\pi(b))\leq0<v(b),$ so that $v(a/b)+\val(\pi(b))=v(a)-v(b)+\val(\pi(b))\leq 0.$ This implies that 
        \begin{align*}
            0&<\val(\pi(a))&\text{ (because }a\in Q)\\
           &=\val\left(\pi\left(\dfrac{a}{b}b\right)\right)&\\
           &=\val\left(\dfrac{a}{b}\cdot\pi(b)\right)&\text{ (because }\dfrac{a}{b}\in Q)\\
           &=v(a/b)+\val(\pi(b))&\text{ (because }v(a/b)+\val(\pi(b))\leq 0)\\
           &\leq 0,
        \end{align*}
        which is absurd. Since $b\in\M$ and $q\in k_0^{\times}\subseteq\O^{\times},$ we have that $b-q\in\O^{\times},$ and by Lemma \ref{valinv}, $\val(\pi(b))=\val(\pi(b-q))=\val\left(\pi\left(\dfrac{1}{b-q}\right)\right).$ Thus,
    \begin{align*}
        0&<\val\left(\pi\left(\dfrac{a}{b-q}\right)\right) & \text{ (because }\dfrac{a}{b-q}\in Q)\\
        &=\val\left(a\cdot\pi\left(\dfrac{1}{b-q}\right)\right)&\text{ (because }a\in Q)\\
        &=v(a)+\val\left(\pi\left(\dfrac{1}{b-q}\right)\right)& \text{ (because }v(a)+\val\left(\pi\left(\dfrac{1}{b-q}\right)\right)=v(a)+\val(\pi(b))\leq 0)\\
        &=v(a)+v(\pi(b))&\\
        &< 0,
    \end{align*}
    which is absurd. 
    \qedhere
    \end{enumerate}
\end{proof}

\begin{propo}[Cf.~{\cite[Proposition 8.30]{dp4}}]
\label{nomutation}
    Let $K$ be a field of positive characteristic that admits dense pre-diffeo-valued data, and let $\s:\Dir_K(K)\to\Dir_k(k^{2})$ be its diffeo-valuation inflator. Then no fundamental ring of a mutation of $\s$ contains a non-zero ideal of a good multi-valuation ring. 
\end{propo}

Note that the conclusion of the proposition is equivalent to saying that no mutation of $\s$ is of $V^{|k_0|}$-type.

\begin{proof}
    Otherwise, suppose $R_{\s'}$ is the fundamental ring of a mutation $\s'$ of $\s$ and suppose $S$ is a good multi-valuation ring of $K$ such that $aS\subseteq R_{\s'}$ for some non-zero $a\in S.$ We know that $\s$ is the inflator associated to the pedestal $I$ of $\Sub_Q(K).$ If $\s'$ is the mutation along the line $K\cdot(a_1,\.,a_n),$ then $\s'$ is the inflator associated to the pedestal $I'=a_1^{-1}I\cap\.\cap a_n^{-1}I,$ with fundamental ring $R_{\s'}=\Stab(I')=\{x\in K:xI'\subseteq I'\}.$ Then $aS\subseteq R_{\s'}$ implies that $abS\subseteq bR_{\s'}\subseteq I'$ for some non-zero $b\in I',$ so if $a_i\neq 0,$ then $a_ibaS\subseteq a_ibR_{\s'}\subseteq a_iI'\subseteq I\subseteq Q.$ But then $a_iba=0$ by Lemma \ref{nogood}, which is absurd.
\end{proof}

\begin{teorema}[Cf.~{\cite[Theorem 8.31]{dp4}}]
\label{t1}
    Let $\s$ be a 2-inflator on a field of positive characteristic. The following statements are equivalent.
    \begin{enumerate}[wide]
        \item $\s$ is a diffeo-valuation inflator.

        \item $\s$ is malleable, isotypical and no mutation of $\s$ is of weak $V^{|k_0|}$-type.
        
        \item $\s$ is malleable, isotypical and no mutation of $\s$ is of weak $V^2$-type.
    \end{enumerate}
\end{teorema}
\begin{proof}
    $(1)\implies(2):$ If $\s:\Dir_K(K)\to\Dir_k(k^{2})$ is a diffeo-valuation inflator, then it is malleable by definition, isotypical because its target is $\Dir_k(k^{2})$ and no mutation of $\s$ is weakly of $V^{|k_0|}$-type by Proposition \ref{nomutation}.

    $(2)\implies(3):$ If some mutation $\s'$ of $\s$ is of weak $V^2$-type, then there is a non-zero ideal $J$ of a multi-valuation ring $S$ of $K,$ with at most 2 maximal ideals, such that $J\subseteq R_{\s'}.$ Then $S$ has at most $2\leq|k_0|$ maximal ideals, so $\s'$ is of weak $V^{|k_0|}$-type.  

    $(3)\implies(1):$ If $\s$ is malleable, isotypical and none of its mutations is of weak $V^2$-type, then $\s$ is a diffeo-valuation inflator by Theorem \ref{dvi}. 
\end{proof}

\subsection{\texorpdfstring{$W_2$}{}-topologies}

In this section we work with fields of characteristic $p>0,$ and we still denote $k_0=\F_p^{alg}\cap K.$ Let $n\geq0$ and let $R$ be a ring. If its lattice of $R$-submodules $\Sub_R(R)$ has cube rank at most $n,$ we say that $R$ is a \emph{$W_n$-ring}. We say that a ring topology $\t$ on a field $K$ is a \emph{$W_n$-topology} if $(K,\t)$ is locally equivalent to a field $L$ with a topology \emph{induced} by a $W_n$-subring $S$ such that $\Frac(S)=L$, cf.~\cite[Corollary 3.8]{dp5}, meaning that the set of non-zero (principal) ideals of $S$ form a local basis of 0. Finally, we say that a ring topology on a field is \emph{definable} if there is a uniformly definable local basis around 0.
For a review of model theory of topological fields, local equivalence and $W$-topologies, see \cite{pz} and \cite{dp5}.

\begin{defin}
    A \emph{$DV$-topology} $\t$ on a field $L$ is a topological field $(L,\t)$ locally equivalent to $(K,\t_R)$ where $(K,\O,D,R,\phi,\dd,\pi)$ is some dense pre-diffeo-valued field and $\t_R$ is the topology induced by $R.$ We say that $(L,\t)$ is a \emph{$DV$-topological field.}  
\end{defin}

\begin{remark}
\label{proptop}
    Let $(K,\O,D,R,\phi,\dd,\pi)$ be a dense pre-diffeo-valued field. Lemma \ref{ryq} and Lemma \ref{top} show that $Q\subseteq R$ and if $u\in I$ is non zero, then $uR\subseteq I\subseteq Q,$ so $R$ and $Q$ are co-embeddable, meaning that they induce the same topology. Such a topology is a Hausdorff, non-discrete, locally bounded field topology on $K$, because $Q$ is a local ring which is proper in $\Frac(Q)=K.$ A posteriori, all $DV$-topologies satisfy these properties. 
\end{remark}

\begin{propo}[Cf.~{\cite[Proposition 8.17]{dp4}}]
\label{dvnotv}
    No $DV$-topology is a $V$-topology.
\end{propo}
\begin{proof}
    It is enough to show that if $(K,\O,D,R,\phi,\dd,\pi)$ is a dense pre-diffeo-valued field and $\t$ is the topology induced by $R,$ then $\t$ is not a $V$-topology. 
    If $\t$ is a $V$-topology, then $(K,\t)$ satisfies the following local sentence:
    \begin{equation}
        \label{vtop}
        \A U\in\t\,\E V\in\t\,\A x,y\in K\,(xy\in V\to(x\in U\vee y\in U)).
    \end{equation} 
    Let $U=R.$ Then there is some non-zero ideal $J$ of $R$ playing the role of $V$ in sentence \ref{vtop}. We may assume that $J=aR$ for some non-zero $a\in R.$ Use Property \ref{7} and surjectivity of $\pi$ to find some $x'\in\O$ such that $\val(\pi(x'))<0.$ Use density of $Q$ to find some $q\in Q$ and $x\in\M$ such that $x'=q+x$ and $v(x)>v(a).$ Then $\val(\pi(x))=\val(\pi(x'))<0.$ We get a contradiction with $x$ and $y=a/x,$ because $xy=a\in aR$ but $x\not\in R$ as $\val(\pi(x))$ is negative, and $y=a/x\not\in R$ because $v(a/x)$ is negative.     
\end{proof}

\begin{propo}[Cf.~{\cite[Theorem 8.3]{dp5}}]
\label{w2isdvorvn}
    If $R$ is a $W_2$-ring on a field $K$ of positive characteristic, then either 
    \begin{enumerate}[wide]
        \item $R$ is co-embeddable with a ring $R_\pi=\{x\in\O:v_\pi(x)\geq 0\}$ given by some dense pre-diffeo-valuation data on $K,$ or

        \item $R$ is co-embeddable with some multi-valuation ring $S$ of $K.$ In particular, there is some $a\neq0$ such that $aS\subseteq R.$
    \end{enumerate}
\end{propo}
\begin{proof}
    Let $m=\wt(R)\leq2.$ Let $\s$ be the malleable $m$-inflator induced by a semi-simple quotient of length $m$ in $\Sub_R(K),$ so that if $\s'$ is a mutation of $\s$ and $R_{\s'}$ is its fundamental ring, then $R$ and $R_{\s'}$ are co-embeddable, cf.~\cite[Proposition 8.2]{dp5}. There are two cases:
    \begin{enumerate}[wide]
        \item If some mutation $\s'$ of $\s$ is weakly of multi-valuation type, then there is some multi-valuation ring $S$ of $K$ and some non-zero $b\in K$ such that $cS\subseteq R_{\s'}.$ There is some non-zero $c\in K$ such that $bR_{\s'}\subseteq R,$ so $bc\neq 0$ and $bcS\subseteq bR_{\s'}\subseteq R,$ as wanted.

        \item If no mutation of $\s$ is weakly of multi-valuation type, then $m=2$ because fundamental rings of malleable 1-inflators are valuation rings, cf.~\cite[Proposition 5.32]{dp3}. In this case, there is a mutation $\s'$ of $\s$ and dense-pre-diffeo-valuation data  
        \begin{center}
            \begin{tikzcd}
        \O \arrow[rr, "\pi"]                 &  & D                     \\
                                             &  &                       \\
        R_{\s'} \arrow[uu, hook] \arrow[rr, "\dd"] &  & k \arrow[uu, "\phi"']
        \end{tikzcd}
        \end{center}
        on $K.$ Lemma \ref{ryq} shows that $R_{\s'}=R_\pi,$ so $R$ is co-embeddable with $R_\pi$ as wanted.
        \qedhere
    \end{enumerate}
\end{proof}

\begin{propo}[Cf.~{\cite[Lemma 8.4]{dp5}}]
\label{dvw2notvn}
    Let $K$ be a field of positive characteristic, and let $\t$ be a $DV$-topology on $K.$ Then $\t$ is a $W_2$-topology and not a $V^{n}$-topology for any $n\geq1.$ 
\end{propo}
\begin{proof}
    Since $(K,\t)$ is a $DV$-topology and the $W_2$ property is locally axiomatizable, we may assume that there is dense pre-diffeo-valuation data $(\O,D,R,\phi,\dd,\pi)$ on $K$ such that $\t$ is induced by $R.$ By Remark \ref{proptop}, $\t$ is also induced by $Q.$ By Lemma \ref{qw2}, $Q$ is a $W_2$-ring of $K,$ so $\t$ is a $W_2$-topology.  
    Now, suppose that $\t$ is a $V^{n}$-topology for some $n\geq1.$ Passing to a sufficiently saturated elementary extension $(K^{*},\O^{*},D^{*},R^{*},\phi^{*},\dd^{*},\pi^{*})$ of $(K,\O,D,R,\phi,\dd,\pi),$ we get that $(K^{*},\t^{*})$ is locally equivalent to $(K,\t).$ By saturation, we may assume that $\t^{*}$ is induced by a multi-valuation ring $S$ of $K^{*}$ having at most $n$ maximal ideals. Then $S$ is $\t^{*}$-bounded, so, by saturation, there is some $W_2$-overring $S'$ of $S$ inducing $\t^{*},$ cf.~\cite[Lemma 3.7]{dp5}. Therefore $S'$ itself is a multi-valuation ring with at most 2 maximal ideals. It cannot have one maximal ideal, otherwise it would be a valuation ring and $\t^{*}$ would be a $V$-topology, in contradiction with Proposition \ref{dvnotv}. If $k_0^{*}=\F_p^{alg}\cap K^{*},$ then $|k_0^{*}|\geq2$ and we get that $S'$ is a good multi-valuation ring co-embeddable with $R^{*}.$ This implies that there is some non-zero $a\in K^{*}$ such that $aS'\subseteq R^{*},$ in contradiction with Lemma \ref{nogood}.
\end{proof}

\begin{defin}
We say that a topological field $(K,\t)$ is \emph{strictly $V^{n}$} or that $\t$ is a \emph{strict $V^{n}$-topology} if $(K,\t)$ is $V^{n}$ but not $V^{m}$ for any $m<n.$     
\end{defin}

\begin{teorema}[Cf.~{\cite[Theorem 8.5]{dp5}}]
\label{classw2}
    It $\t$ is a $W_2$-topology, then exactly one of the following holds:
    \begin{enumerate}[wide]
        \item $\t$ is a $V$-topology,

        \item $\t$ is a strict $V^{2}$-topology,

        \item $\t$ is a $DV$-topology.
    \end{enumerate}
    Moreover, all such topologies are $W_2$-topologies.
\end{teorema}
\begin{proof}
    If the characteristic of the underlying field is 0, then this is proven in \cite[Theorem 8.5]{dp5}. Thus, we may assume that the characteristic of the underlying field is positive. $V$- and $V^{2}$-topologies are $W_2$-topologies. $DV$-topologies are $W_2$-topologies too, by Proposition \ref{dvw2notvn}. If $\t$ is a $W_2$-topology on a field $K$, then upon saturating we may assume that it is induced by some $W_2$-ring $R$ of $K.$ By Proposition \ref{w2isdvorvn}, since $\char(K)>0,$ we get that either $R$ is co-embeddable with some $R_\pi$ given by dense pre-diffeo-valuation data on $K,$ or that $R$ is co-embeddable with some multi-valuation ring of $K.$ The first case implies that $\t$ is a $DV$-topology, and the second case implies that $\t$ is $V^{n}$ for some $n\geq 1.$ As in the proof of Proposition \ref{dvw2notvn}, $\t$ has to be either a $V$-topology or a $V^{2}$-topology. Finally, all three cases are exclusive by Propositions \ref{dvnotv} and \ref{dvw2notvn}.
\end{proof}

We conclude several corollaries from Theorem \ref{classw2}.

\begin{cor}
\label{perfw2}
Let $K$ be a field of positive characteristic.
    \begin{enumerate}[wide]
        \item If $K$ admits a $DV$-topology $\t$, then $K$ is not perfect.

        \item If $\t$ is a $W_2$-topology on $K,$ then $\t$ is a  $V^{2}$-topology. 
    \end{enumerate}
\end{cor}
\begin{proof}
\begin{enumerate}[wide]
    \item We may assume that $\t$ is induced by some dense pre-diffeo-valuation data $(\O,D,R,\dd,\pi,\phi)$ in $K.$ Then, if $\char(K)=p>0,$ we know that $\O^{p}\subseteq Q\subseteq R\subseteq\O.$ If $K$ is perfect, then $\O=\O^{p}$ and $\dd$ is trivial, which is absurd. 

    \item This follows from the classification given by Theorem \ref{classw2}.
    \qedhere
\end{enumerate}
\end{proof}

\section*{Acknowledgments}
The author would like to thank Will Johnson for bringing Yang Yang's master's thesis to his attention. The author has received funding from the European Union’s Horizon 2020 research and innovation program under the Marie Sk\lw odowska-Curie grant agreement No 945332. 
\includegraphics*[scale = 0.028]{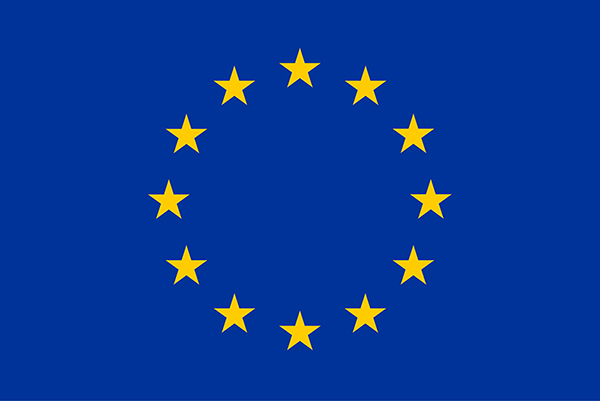}

\bibliographystyle{alpha}
\bibliography{bib}

\end{document}